%% file: unref_nontriang.tex
\documentclass[oneside,reqno]{amsart}
\input{preamble}

\usepackage{orcidlink}

\begin{document}
\title[]{The number of maximal unrefinable partitions} 
 \author{Riccardo Aragona$^1$\orcidlink{0000-0001-8834-4358}, Lorenzo Campioni$^1$, Roberto Civino$^1$\orcidlink{0000-0003-3672-8485}}
 
 \address{$^1$
 DISIM - Universit\`a degli Studi dell'Aquila, Italy}    

\email{riccardo.aragona@univaq.it, lorenzo.campioni1@univaq.it, roberto.civino@univaq.it}

%\date{} %\thanks{}


\begin{abstract}
This paper completes the classification of maximal unrefinable partitions, extending a previous work
of Aragona et al.\ devoted only to the case of triangular numbers. We show that the number of maximal 
unrefinable partitions of an integer coincides with the number of suitable partitions into distinct parts, 
depending on the distance from the successive triangular number.
\end{abstract}

\maketitle

%%%%%%%%%%%%%%%%%%%%%%%%%%%%%%%%%%%%%%%%%%%%%%%%%%%%%%%%%%%%%%
%%%%%%%%%%%%%%%%%%%%%%%%%%%%%%%%%%%%%%%%%%%%%%%%%%%%%%%%%%%%%%
%%%%%%%%%%%%%%%%%%%%%          S E C T I O N 1         %%%%%%%%%%%%%%%%%%%%%%%%%%%
%%%%%%%%%%%%%%%%%%%%%%%%%%%%%%%%%%%%%%%%%%%%%%%%%%%%%%%%%%%%%%
%%%%%%%%%%%%%%%%%%%%%%%%%%%%%%%%%%%%%%%%%%%%%%%%%%%%%%%%%%%%%%

\section{Introduction}
Let $N \in \mathbb N$. A partition of $N$ into distinct parts is called \emph{unrefinable} if none of its parts $x$ can be replaced by  integers
 whose sum is $x$ and which do not already belong to the partition. Since such definition naturally imposes a limitation on the size of the largest part, 
we call \emph{maximal unrefinable} partitions those where the largest part is maximal among those having the same sum. Aragona et al.~\cite{aragona2021maximal}
have recently shown that if $N$ is the triangular number $T_n = \binom{n+1}{2}$ and $\listP$ is an unrefinable partition of $N$, then $\l_t \leq 2n-4$
and that the bound is sharp.
Moreover, they calculated that the number of unrefinable partitions of $T_n$ attaining the bound, i.e., those where $\l_t=2n-4$, is as follows:
\begin{notheorem}[\cite{aragona2021maximal}]
Let $n \in \mathbb N$. The number of maximal unrefinable partitions of $T_n$ is one if $n$ is even and coincides with the number of partitions
 of $(n+1)/2$ into distinct parts if $n$ is odd.
\end{notheorem}
The aim of this work is to complete the classification of maximal unrefinable partitions, extending the previous result to the case of non-triangular numbers.
If $N \in \mathbb N$ is non-triangular, then it is uniquely determined by a pair $(n,d)$ where $n \in \mathbb N$ and $1 \leq d \leq n-1$ such that
$N=T_n-d$. We denote such an integer $N$ by $T_{n,d}$. Before introducing the main contribution of this paper related to the number of maximal
unrefinable partitions of $T_{n,d}$, let us specify the notation.
\subsection{Notation}
The sequence of positive integers $\listP$ is a \emph{partition of $N$ into distinct parts} if $\sum \l_i = N$, $\l_1 < \l_2 < \dots < \l_t$ and $t \geq 2$. We write $\l \vdash N$ and $|\l| = t$. The set $\dist_N$ denotes the set of all the partitions of $N$ into distinct parts,
while $\mathbb{D}_{N,s}$ denotes the set 
$\mathbb{D}_{N,s}=\left\{\l\in\mathbb{D}_{N}\mid \left|\l\right|=s \right\}$ of partitions of $N$  into $s$ distinct  parts (or of length $s$). 
 Moreover, $\mathbb{D}_{N}^{\,\text {odd }}$ denotes the subset of $\dist_N$ composed of partitions in which each part is odd.
%, and the number $p_\dist(N) = \#\dist_N$ is the numbers of partitions of $N$ into distinct parts.
If $\listP \in \dist$, the integers belonging to \[\mathcal M_\l \deq \{1,2,\dots, {\l_t}\} \setminus \{\l_1,\l_2,\dots, \l_t\}\] are called the \emph{missing parts} of $\lambda$, and are denoted by $\mu_1 < \mu_2 < \dots <\mu_m$, for some $m \geq 0$.

%We say that a partition of $N$ into distinct parts is refinable if one of its parts can be replaced by two different integers such that the resulting partition is still a partition of $N$ into distinct parts. More precisely:

It is not hard to realize that if a partition is refinable, then its smallest refinable part has a refinement of the form $a +b$~\cite[Proposition~4]{aragona2021verification}. This justifies the following definition.
\begin{definition}\label{def_unref}
Let $\listP$ be a partition of $N$ into distinct parts and let $\mu_1 < \mu_2 < \dots <\mu_m$ be its missing parts.  The partition $\l$ is \emph{refinable} if there exist $1 \leq \ell \leq t$ and $1\leq i<j \leq m $  such that $\mu_i+\mu_j = \l_{\ell}$, and \emph{unrefinable} otherwise. The set of unrefinable partitions of $N$  is denoted by $\up_N$.
An unrefinable partition $\listP \in \up_N$ is called \emph{maximal} if  
\[
\l_t = \max_{(\l_1', \l_2', \dots, \l_t') \in \up_{N}} \l_t'.
\]
We denote by ${\mup{N}}$ the set of the maximal unrefinable partitions of $N$.
\end{definition}

If $A$ and $B$ are sets and $\phi \colon A \rightarrow B$ is a bijection, then we write $A \overset{\phi}{\leftrightarrow} B$ to mean that $\phi$ sends bijectively $A$ into $B$. We denote by $\#A$ the number of elements of $A$.

\subsection{Contributions}\label{sec:contri}
As already anticipated, the aim of this work is to construct and count maximal unrefinable partitions for non-triangular numbers $T_{n,d}$, for $n \in \mathbb N$ and $1 \leq d \leq n-1$.
For a fixed $n \in \mathbb N$ and for $1 \leq d \leq n-1$, we show a sharp upper bound for $\l_t$ when $\listP \in \up_{T_{n,d}}$. In particular, we show that, if $3 < d \leq n-1$, then 
\begin{equation}\label{eq:bounds}
\begin{cases}
\l_t \leq 2n-5 & \text{if } n-d \text{ is even},\\
\l_t \leq 2n-4 & \text{if } n-d \text{ is odd}.
\end{cases}
\end{equation}
Moreover, in the extremal cases $d \in \{1,2,3\}$, we obtain respectively:  $\l_t\leq 2n-2$, $\l_t\leq 2n-3$ and $\l_t\leq 2n-4$.
Once it is clear what \emph{maximal} means for an unrefinable partition of $T_{n,d}$, we can construct, classify and count all
the partitions of $\mup{T_{n,d}}$, by way of a bijective proof. As in the case of triangular numbers, maximal unrefinable partitions can be expressed in terms of suitable partitions
into distinct parts, for sufficiently large $n$. We will show the following result.
\begin{mainres}
Let $n \in \mathbb N$, $n \geq 11$, and $1\leq d \leq n-1 $. If $d >3$, the number of maximal unrefinable partitions of $T_{n,d}$ is
\[
\#\mup{T_{n,d}}=
\begin{cases}
\# {\mathbb{D}_{n-d+2}^{\,\text {odd }}}& \text{ if } n-d \text{ is even},\\
1+\#\dist_{(n-d+1)/2} & \text{ if } n-d \text{ is odd}.
\end{cases}
\]
Otherwise, for $d \in \{1,2\}$ we have $\#\mup{T_{n,d}}=1$ and for $d=3$
\[
\#\mup{T_{n,3}}=
\begin{cases}
\#\dist_{(n-2)/2} & \text{if } n \text{ is even},\\
1 & \text{if } n \text{ is odd}.
\end{cases}
\]
\end{mainres}
\subsection{Related works}
The concept of \emph{unrefinability} of a partition, formalized here in Definition~\ref{def_unref}, comes from a quite natural 
constraint among the parts, therefore the authors' belief is that the corresponding combinatorial object is intrinsically interesting. Nonetheless,
only few partial results are known on this topic
(cf.~the On-Line Encyclopedia of Integer Sequences for the first values of the sequence of unrefinable partitions~\cite[\url{https://oeis.org/A179009}]{OEIS}).
However, in addition to the purely theoretical interest, it has been shown that unrefinable partitions are surprisingly related to the generators of the 
$(n-1)$-th term in a chain of normalizers in the Sylow $2$-subgroups of $\Sym(2^n)$~\cite{aragona2021unrefinable}, while the generators
of the previous $n-2$ terms of the chain were linked to partitions into distinct parts in the classical sense~\cite{Aragona2019,aragona2021rigid}.

We have already explained why the size of the largest part, in this case, is an interesting statistic. First results on the classification of maximal unrefinable partitions have been obtained by Aragona et al.\ in the case of triangular numbers~\cite{aragona2021maximal}, combining a property of symmetry, which is similar to the one described later in this paper, and a bound in the minimal excludant~\cite{fraenkel2015harnessing}, which has been  investigated also recently by other authors~\cite{andrews2019,Hopkins2022}.
With similar techniques and by means of a polynomial-time algorithm for the generation of all the unrefinable partitions of a given integer used for simulations~\cite{aragona2021verification}, we complete here the classification of maximal unrefinable partitions extending previous constructions to the case of non-triangular numbers.

\subsection{Organization of the paper}
We show the bound of Eq.~\eqref{eq:bounds} and those related to the cases $d \in \{1,2,3\}$ in Sec.~\ref{sec:bounds}. The bounds are obtained constructively, i.e., we show actual partitions which attain the bounds. Such constructions are then extended in Sec.~\ref{sec:count} in a complete classification of maximal unrefinable partitions attaining the corresponding bounds. With similar arguments but slightly different computations, we  address the cases $\l_t \leq 2n-4$ and $\l_t \leq 2n-5$ in two separate subsections, i.e., respectively in Sec.~\ref{sec:2nminus4} and in Sec.~\ref{sec:2nminus5}. In particular,
the already mentioned counting result proved by a bijective argument can be read in Theorem~\ref{thm:maindodd} and in Theorem~\ref{thm_main2}. We draw our conclusions in Sec.~\ref{sec_final}.

\section{Construction of maximal unrefinable partitions}\label{sec:bounds}
From now on, let us assume $n \geq 11$. In this section we show that the bound for the largest part in an unrefinable partition for a non-triangular number depends on the parity of the distance from the index of the successive triangular number. To do this, we start from the convenient partition of Definition~\ref{def:unref}.
For each non-triangular integer $N > 3$, indeed, there exists an unrefinable partition of $N$ determined by the successive triangular number $T_n$ and by the distance $d$ of $N$ from it: 
\begin{definition}\label{def:unref}
Let $n \in \mathbb N$ and $1\leq d \leq n-1 $. We denote by $\pi_{n,d}$ the unrefinable partition of $T_{n,d}$ defined by 
\[
\pi_{n,d} \deq (1,2,\dots, d-1,d+1,\dots,n) \vdash T_{n,d}.
\]
\end{definition}
In order to construct other partitions of $T_{n,d}$, we will proceed as follows: starting from $\pi_{n,d} \in T_{n,d}$, we create a new partition $\l$  by removing from $\pi_{n,d}$ some of its parts, namely $a_1, a_2, \dots, a_h \in \{1,2,\dots,d-1,d+1,\dots,n\}$, and, at the same time, by adding to $\l$ new parts $\al_1, \al_2, \dots, \al_j \in \{d\,\}\cup\{
s \mid s \geq n+1\}$, for some positive integers $h$ and $j$. This leads to the creation of a partition $\l \vdash T_{n,d}$ if $j \leq h$ and
\[
\sum_{i=1}^{h} a_i = \sum_{i=1}^{j} \al_i.
\]
The obtained partition may be, in general,  refinable or not.
This notation will be used in the remainder of the paper and this strategy, in addition to further unrefinability checks, will lead to the classification of $T_{n,d}$.

Notice that when $\l$ is an unrefinable partition of $\mup{T_{n,d}}$, the missing part $d$ in $\pi_{n,d}$ can be either one of the \emph{replacements} $\al_i$s or not. 
We will show in Proposition~\ref{prop:maxlt} that, depending on this, we will obtain two different bounds.
The proof of this first result relies on the following known fact~\cite{aragona2021maximal}, which we reproduce here for completeness. 

\begin{lemma}[\cite{aragona2021maximal}]\label{lemma:bound}
Let $\listP$ be unrefinable and let $\missP$ be its missing parts. Then the number $m$ of the missing parts is bounded by
\begin{equation*}\label{bound1}
m \leq \left\lfloor \dfrac{\l_t}{2}\right\rfloor.
\end{equation*}
\end{lemma}

\begin{proof}
Let us start by observing that $\l_t-\mu_i \in \l$ for $1 \leq i \leq m$, otherwise from $\l_t-\mu_i, \mu_i \in \Ml$, we obtain $ (\l_t-\mu_i)+\mu_i=\l_t\in \l$ and thus $\l$ is refinable.
Only the case $\mu_i = \l_t/2$ when $\l_t$ is even is exceptional.
We prove the claim by removing from the partition $(1,2,\dots,\l_t)$ the maximum number of parts different from $\l_t$. 
For the previous observation, each candidate part $\mu_i$ to be removed  has a counterpart $\l_t-\mu_i$ in the partition.
The bound in the statement depends on the fact that this process can be repeated no more than $\lfloor {\l_t}/{2}\rfloor$ times.
\end{proof}

\begin{proposition}\label{prop:maxlt}
Let $N=T_{n,d}$ with $n \in \mathbb N$ and $1\leq d \leq n-1$, and let $\listP\in\mup{N}$. Then 
\[
\l_t\leq 
\begin{cases}
2n-2 & d\in\l,\\
2n-4 & d\not\in\l.
%2n-4 & d\in\list.
\end{cases}
\]
\end{proposition}
\begin{proof}
We will use, here and in the following proofs, the notation introduced after Definition~\ref{def:unref}.
Let us first assume that $d\not\in\l$, then the number of missing parts of $\l$ is $h+1+(\l_t-n-j)$. From Lemma~\ref{lemma:bound}, we have $h+1+\l_t-n-j\leq\lfloor\l_t/2\rfloor$ and so
\begin{align}
h&\leq \lfloor\l_t/2\rfloor - \l_t+n+j-1 \nonumber \\
&=n+j-1-\lceil\l_t/2\rceil. \label{eq:maxh}
\end{align}
If $d\not\in\l$, then we have $j < h$ and so, from Eq.~\eqref{eq:maxh}, $n-1-\lceil \l_t/2\rceil\geq 1$, i.e., $\l_t\leq 2n-4$.
If we assume $d\in\l$, from Eq.~\eqref{eq:maxh} and from $j\leq h$ we have $n-1-\lceil \l_t/2\rceil\geq 0$, i.e., $\l_t\leq 2n-2$.
\end{proof}

\begin{remark}\label{rmk1}
First notice that $1\leq a_1 <a_2<\cdots < a_h\le n$, $d<n$ and $n+1\leq \al_i \leq \l_t$ for each $1\leq i \leq j$ such that $\al_i\ne d$, hence, since $\su$, if $d\in\l$, then $j\leq h$, otherwise $j<h$. In particular if $\l_t\in\{2n-3,2n-2\}$, which is by Proposition~\ref{prop:maxlt} only possible when $d\in\l$, we have $j\leq h$. Moreover, from Eq.~\eqref{eq:maxh} we have $h\leq j+n-1-\lceil \l_t/2\rceil=j$. In other words, if $\l_t\in\{2n-3,2n-2\}$, then $h=j$. In the case when $\l_t=2n-4$, if $d\in\l$  we have $j\leq h$ and, from Eq.~\eqref{eq:maxh}, $h\leq j+1$, and so $h\in\{j,j+1\}$; instead if $d\not\in\l$, since $j<h$, we obtain $h=j+1$.
\end{remark}
\begin{remark}[Anti-symmetry]\label{rmk2}
Let  $\pi_{n,d} \neq \listP\vdash N$ be unrefinable. In this case $\l_t \geq n+1$. Notice that if an integer $x$ in $\{1,\dots, \l_t-n-1\}$ is such that $x\not\in \l$, then it corresponds to an element $x'=\l_t-x\in\{n+1,\dots,\l_t-1\}$ such that $x'\in\l$, otherwise $x+x'=\l_t$ and $\l$ is refinable.  Therefore, the parts of $\l$  can belong to three consecutive \emph{areas} of $\{1,2,\ldots,\l_t-1\}$, as shown in Fig~\ref{fig:areas}. We call
\begin{itemize}
\item the\textit{ first area} the set $\{s \in \mathbb N \mid 1 \leq s \leq \l_t-n-1\}$,
\item the\textit{ free area} the set $\{s \in \mathbb N \mid \l_t-n \leq s \leq n\}$,  
\item the\textit{ last area} the set $\{s \in \mathbb N \mid n+1 \leq s < \l_t\}$.
\end{itemize} 
Choosing elements in the first area implies fixing parts in the last one. For this reason, if we consider $\pi_{n,d}$ and if we obtain a new unrefinable partition $\l \vdash T_{n,d}$ from $\pi_{n,d}$ removing $a_1,a_2,\ldots,a_h\leq n$ and replacing them with $\al_1,\al_2,\ldots,\al_j$, then each $a_i$ in the first area determines $\l_t-a_i\in\{\al_i\}_{i=1}^{j-1}$. Accordingly, we denote the element $\l_t-a_i$ by $\al_i$. In particular $\l_t=\al_j$ and, when $d\in\l$, we denote $d$ by $\al_{j-1}$.

\begin{figure}[h!]
 \[
    \begin{tikzpicture}
    \draw (0,0) -- (8,0);
    \coordinate[label=below:$\color{white}{\dot{n}}\color{black}1$] (1) at (0,0);
    \coordinate[label=below:$\l_t-n-1$] (n-4) at (2.75,0);
    \coordinate[label=below:$ \color{white}{\dot{n}} \color{black}n $] (n) at (4,0);
    \coordinate[label=below:$\color{white}{1} \color{black} \l_t$] (2n-4) at (8,0);
    \fill (1) circle (2pt) (n-4) circle (2pt);
    \fill (n) circle (2pt) (2n-4) circle (2pt);
    \node[align=center] at (1.5,-1) {first \\ area};
    \node[align=center] at (3.4,-1) {free \\ area};
    \node[align=center] at (6,-1) {last \\ area};
    \end{tikzpicture}
    \]
        \caption{The three areas of the parts in an unrefinable partition}
    \label{fig:areas}
    \end{figure}
    
\end{remark}

By Proposition~\ref{prop:maxlt} we know that if $\l \in \mup{T_{n,d}}$, then $\l_t\leq 2n-2$. In the following sections, we will  distinguish all the possible cases for $\l_t$ and we will provide the corresponding constructions.
\subsection{The case $\l_t=2n-2$}
By virtue of Proposition~\ref{prop:maxlt} we know that if $\l \in \mup{T_{n,d}}$, then $d\in\l$ implies $\l_t \leq 2n-2$. Let us now show that the bound is attained sharply only for a single choice of $d$. 
\begin{proposition}\label{prop:d1}
Let $\listP \in \mup{T_{n,d}}$. If $\l_t=2n-2$, then $d=1$ and such a partition is unique.
\end{proposition}
\begin{proof}
From Proposition~\ref{prop:maxlt} we have that $\l_t=2n-2$ implies $d \in \l$ and by Remark~\ref{rmk1} we also know that $h=j \geq 2$. From the hypothesis $\l_t=2n-2$ we obtain that the free area corresponds to the set $\{n-2,n-1,n\}$. By Remark~\ref{rmk2} we have that $\al_{h-1}=d$ and $\al_h=2n-2$ are fixed.  Therefore, since $h=j$, the free area can contain two or three parts, but we must rule out the second option since it would violate unrefinability. We are then only left with the case of two parts chosen among $\{(n-2,n-1),(n-1,n)\}$. The case $(n-2,n)$ is not considered since $n-2+n=2n-2$ and $\l$ is unrefinable. Let us distinguish all the possible cases for $h$.

Let $h=2$. Since $\al_1$ and $\al_2$ are already fixed, we have that $a_1$ and $a_2$ are free elements. From $\su$ we have that either 
\begin{align*}
(n-2)+(n-1)&=d+2n-2 \quad \text{or}\\
(n-1)+n&=d+2n-2.
\end{align*}
From the first equation we obtain $d=-1$, a contradiction. From the second one we obtain $d=1$, as claimed. Indeed the obtained partition 
$\l=(1,2,\dots, n-2, 2n-2)$
is unrefinable since the sum of the first two missing parts $n-1$ and $n$ is larger than $\l_t=2n-2$.
Let us now prove that the remaining cases lead to contradictions.

Let $h=3$. From the hypothesis and from Remark~\ref{rmk2}, we have $\al_3=2n-2$, $\al_2=d$, $\al_1=2n-2-a_1$ is 
determined by the choice of $a_1$ and $(a_2,a_3) \in \{(n-2,n-1),(n-1,n)\}$.
Let us assume that $(a_2,a_3) =(n-1,n)$. Then, by $\su$, we obtain
\[
a_1+n-1+n = 2n-2-a_1+d+2n-2,
\]
from which 
\[
a_1 = \frac{2n-3+d}{2}.
\]
By checking if $a_1 \leq n-3$ as it should be, we determine a contradiction on $d$.
The other option for $(a_2,a_3)$ corresponds to a larger value for $a_1$, even more so a contradiction. 

Let $h\geq4$. We are assuming $\al_h=2n-2$, $\al_{h-1}=d$, $(a_{h-1},a_h) \in \{(n-2,n-1),(n-1,n)\}$ and 
$1 \leq a_1 < a_2 \dots < a_{h-2} \leq n-3$, which determine $\al_1,\al_2, \dots, \al_{h-2}$ as $\al_i = 2n-2-a_i$.
From $\su$ we obtain 
\[
a_1+a_2+\dots+a_{h-2} = \frac{(h-1)(2n-2)+d-(a_{h-1}+a_h)}{2}.
\]
Proceeding as in the previous case, we can choose to maximize $a_2+\dots +a_h$ by setting 
$a_{h-2}=n-3, a_{h-3}=n-4,  \dots, a_2=n-h+1$ and $a_{h-1} = n-1$, $a_h = n$. From this we obtain 
\[
a_1 = \frac{2n+h^2-3h-3+d}{2}.
\]
Checking $a_1<a_2 = n-h+1$, we obtain $h^2-h-5+d < 0$, which is impossible in the current setting where $d >0$ and $h \geq 4$. 
Notice that the choice of maximizing $a_2+\dots +a_h$ leads to the minimum value for $a_1$. Any other choice of $a_2,\dots,a_h$
would lead to a contradiction even more so. %We will use the same argument throughout the proofs of this section, even without further mention.
\end{proof}
Notice that in the previous proof only one construction was successful. Therefore, the following consequence is trivially obtained.
\begin{corollary}
$\#\mup{T_{n,1}}=1.$
\end{corollary}
In the following sections, we will investigate the remaining possibilities for $\l_t$. Notice that we will mimic the arguments of Proposition~\ref{prop:d1}. As before,  given the value of $\l_t$,
we will determine the free area and the number of elements that can be chosen in the free area. Then we will attempt to construct partitions for each possible value of $h$.
In the general case, we will derive the conclusion starting from the choice which maximizes the sum of the values assigned to $a_2, a_3, \dots, a_h$, and minimizes $a_1$. We will use this strategy also in the following proofs, without further mention.

\subsection{The case $\l_t=2n-3$}
\begin{proposition}\label{prop:d2}
Let $\listP \in \mup{T_{n,d}}$. If $\l_t=2n-3$, then $d=2$ and such a partition is unique.
\end{proposition}
\begin{proof}
From the hypothesis and from Remark~\ref{rmk1},  we obtain that $d \in \l$ and $h=j$. 
The free parts are those belonging to $\{n-3, \dots, n\}$. We have already fixed two of the $\al_i$s and it is not possible to choose more that two parts 
in the free area without obtaining a contradiction on the unrefinability of $\l$. Therefore, we are left with the case of two free parts and $h-2$ parts in the first area to be determined.
Only four conditions on $(a_{h-1},a_h)$ do not contradict the unrefinability on $2n-3$, namely 
\[(a_{h-1},a_h) \in \{(n-3,n-2),(n-3,n-1),(n-2,n),(n-1,n)\}.\]
Let us distinguish the possible cases for $h$.

Let $h=2$. From $\su$ we obtain four equations
\begin{align*}
\left(n-3\right)+\left(n-2\right)=2n-5=d+2n-3,\\
\left(n-3\right)+\left(n-1\right)=2n-4=d+2n-3,\\
\left(n-2\right)+n=2n-2=d+2n-3,\\
\left(n-1\right)+n=2n-1=d+2n-3.
\end{align*}
From the first two equations we obtain the contradiction of $d$ being a negative integer. From the third equation we obtain $d=1$, which means that the partition is not maximal (cf. Proposition~\ref{prop:d1}). From the last one we obtain $d=2$, as claimed.
Notice that the obtained partition $\l=(1,2,\dots,n-2,2n-3)$ is unrefinable since the sum of the least missing parts $n-1$ and $n$ is larger than $2n-3$.
In the remainder of the proof, we will show that the remaining cases lead to contradictions.

Let $h=3$. In the current setting we have $\al_2=d$, $\al_3=2n-3$, $(a_{2},a_3) \in \{(n-3,n-2),(n-3,n-1),(n-2,n),(n-1,n)\}$, $1 \leq a_1 \leq n-4$ and $\al_1=2n-3-a_1$.
Proceeding as usual, let us consider the case where $a_2+\dots+a_h$ is maximal, which corresponds to the choice $a_2=n-1$ and $a_3=n$.
From $\su$ we obtain
\[
a_1 = \frac{2n-5+d}{2},
\]
and checking if $a_1 \leq n-4$ we obtain a contradiction on $d$. 

Let $h\geq 4$. Maximizing $a_2+\dots +a_h$, i.e., setting $a_{h-2}= n-4, a_{h-3}=n-5,\dots, a_2=n-h$, $a_{h-1}=n-1$ and $a_h=n$, from $\su$ we have 
\[
a_1=\frac{2n+h^2-2h-8+d}{2}.
\]
Imposing $a_1 < a_2=n-h$ leads to a contradiction.
\end{proof}

\begin{corollary}
$\#\mup{T_{n,2}}=1.$
\end{corollary}

\subsection{The case $\l_t=2n-4$}
In this case, by Proposition~\ref{prop:maxlt}, we have to consider both cases $d\not\in\l$ and $d\in\l$. Let us start by showing that the first assumption gives only one contribution. 
\begin{proposition}\label{prop:2n-4dno0}
Let $\listP\in\mup{T_{n,d}}$ be such that $d\not\in\l$. If $\l_t=2n-4$, then $d=n-5$ and such a partition is unique.
\end{proposition}
\begin{proof}
We derive the claim by proving the following two statements:
\begin{enumerate}
\item if $d\leq n-5$, then $d=n-5$ and there exists only one partition;
\item no partition exists if $n-4\leq d <n$.
\end{enumerate}
Let us now prove each claim separately.
\begin{enumerate}
\item If $\l_t=2n-4$, then the free area is $\{n-4,\ldots,n\}$ and we have, by Remark~\ref{rmk1}, that $h=j+1$. Moreover, from the fact that $d\not\in\l$ and $d\leq n-5$, or in other words $d$ is outside the free area, we must have $\l_t-d\in\l$ since $\l$ is unrefinable. Hence we are left with $j-2$ parts in the last area to be determined. Now, choosing four parts in the free area would contradict the unrefinability of $\l$. We also obtain a contradiction choosing less than two parts in the free area, i.e., more than $h-2=j-1$ parts in the first area. We conclude we can only choose three parts in the free area. In particular we have only four possible cases, i.e., $(a_{h-2},a_{h-1},a_h)\in\{(n-4,n-3,n-2),(n-4,n-2,n-1),(n-3,n-2,n),(n-2,n-1,n)\}$.

Let $h=3$. From $\su$ we obtain four equations by the four possible options in the free area:
\begin{align*}
  3n-9&=4n-8-d,\\
  3n-7&=4n-8-d,\\
  3n-5&=4n-8-d,\\
  3n-3&=4n-8-d.
\end{align*}
The first three equations lead to a contradiction on $d$ while from the last one we obtain $d=n-5$,  corresponding to the partition 
\[
\l=(1,2,\ldots,n-6,n-4,n-3,n+1,2n-4)
\]
which is unrefinable since, by hypothesis, we have $n\geq 11$.

Let $h=4$. As usual, maximizing $a_2+a_3+a_4$, fomr $\su$ we determine
\[
a_1=\frac{3n-9-d}{2}.
\]
Imposing $a_1<n-4$ we obtain $d>n-1$, a contradiction.

Let $h\geq 5$. Maximizing $a_2+a_2+\cdots+a_h$, from $\su$ we obtain
\[
a_1=\frac{3n+h^2-3h-13-d}{2},
\]
being  meaningful when $a_1<a_2=n-h$, from which we obtain $n-d+(h^2-h-13)<0$, a contradiction if $h\geq 5$.
\item Notice that, since $d\not\in\l$ and $n-4\leq d< n$, we can only choose $a_{h-1}$ and $a_h$ in the free area, being the third spot occupied already by $d$. From $\su$, in this case we have 
\begin{equation}\label{eq:2n-4h5}
a_1+a_2+\cdots+a_{h-2}=\frac{(h-1)(2n-4)-(a_{h-1}+a_h)}{2}.
\end{equation}
Let us now examine each possible choice of $d$. If $d=n-4$ or $d=n-2$, then $a_{h-1}+a_h$ is odd, therefore Eq.~\eqref{eq:2n-4h5} cannot be satisfied. Let us now assume that $d=n-3$. In this case $(a_{h-1},a_h)\in\{(n-4,n-2),(n-2,n)\}$. 

Let $h=3$. Maximizing $a_2+a_3$, we can calculate
\[
a_1=\frac{4n-8-(a_2+a_3)}{2},
\]
and so $a_1>n-5$,  a contradiction.

Let $h\geq 4$. Maximizing $a_2+a_3+\cdots+a_h$, from $\su$ we obtain 
\[
a_1=\frac{2n+h^2-h-12}{2}.
\]
Checking if $a_1<a_2=n-h-1$, we derive that $h^2+h-10<0$, a contradiction. The same contradiction is obtained when $d=n-1$.
 \qedhere
\end{enumerate}
\end{proof}
Let us address the remaining case $d\in\l$. Recall that, in this case, by Remark~\ref{rmk1} we have $h\in\{j,j+1\}$.
\begin{proposition}\label{prop:d3}
Let $\listP\in\mup{T_{n,d}}$ be such that $\l_t=2n-4$ and $d\in\l$. If $h=j$, then $d=3$ and such a partition is unique. If $h=j+1$, then for each $1\leq k\leq \lfloor(n-2)/2\rfloor$ there exists $\l\in\mup{T_{n,d}}$ with $d=n-(2k-1)$ and there does not exist $\l\in\mup{T_{n,d}}$ with $d=n-2k$.
\end{proposition}
\begin{proof}
Let us assume that $h=j$. Since, by Remark~\ref{rmk2}, $\al_{h-1}=d$ and $\al_{h}=2n-4$ are already fixed, then $a_{h-1}$ and $a_h$ are parts of the free area. Notice that the free area cannot contain more than three parts. 

\medskip

Let us first assume that it only contains two parts, i.e., $(a_{h-1},a_{h})\in\{(n-4,n-3),(n-4,n-2),(n-4,n-1),(n-3,n-2),(n-3,n),(n-2,n-1),(n-2,n),(n-1,n)\}$.

Let $h=2$. We must have $a_1+a_2=2n-4+d$. In the case when $(a_1,a_2)=(n-1,n)$ we obtain $d=3$. which is the claim, since the corresponding partition $\l=(1,2,\ldots,n-2,2n-4)$ is unrefinable. If $(a_1,a_2)=(n-2,n-1)$ or $(n-2,n)$, we respectively obtain $d=1$ and $d=2$, which, by Proposition~\ref{prop:d1} and Proposition~\ref{prop:d2}, contradicts the maximality of $\l$. In the remaining cases, we obtain $d\leq 0$ which is a contradiction.

Let $h=3$. From $\su$, we have
\[
a_1=\frac{4n-8-(a_2+a_3)+d}{2}.
\]
Considering the maximal choice $(a_2,a_3)=(n-1,n)$ we obtain $a_1=(2n-7+d)/2<n-4$ when $d< -1$, a contradiction.

Let $h\geq 4$. From $\su$ we obtain
\[
a_1+a_2+\cdots+a_{h-2}=\frac{(h-1)(2n-4)-(a_{h-1}+a_h)+d}{2}.
\]
Maximizing $a_2+\ldots+a_{h}$, we have
\[
a_1=\frac{2n+h^2-h-13+d}{2},
\]
which satisfies $a_1<a_2$ when $h^2+h-11+d<0$, a contradiction when $h\geq 4$. 

\vspace{1.5mm}

Under the assumption that $h=j$, it remains to consider the case of three parts in the free area, i.e., $(a_{h-2},a_{h-1},a_h)\in\{(n-4,n-3,n-2),(n-4,n-2,n-1),(n-3,n-2,n),(n-2,n-1,n)\}$. In this case $a_1,a_2,\ldots,a_{h-3}$ determine $\al_1,\al_2,\ldots,\al_{h-3=j-3}$, while the part  $n+1 \leq \al_{j-2} \leq 2n-5$ is not determined by one of the $a_i$s.

Let $h=3$. We have $a_1+a_2+a_3=2n-4+d+\al_1$, with $n+1\leq \al_1\leq 2n-5$ and each possible choice of the parts in the free area implies that $\al_1 < n+1$, which is a contradiction.

Let $h=4$. From $\su$ we have 
\[
a_1=\frac{2(2n-4)-(a_2+a_3+a_4)+d+\al_2}{2},
\]
and maximizing $a_2+a_3+a_4$ we obtain $a_1=(n-5+d+\al_2)/2$. Since $a_1<n-4$, then $\al_2 < n-3-d < n-3$, a contradiction.

Let $h\geq 5$.  From $\su$ we have
\[
a_1+a_2+\cdots+a_{h-3}=\frac{(h-2)(2n-4)-(a_{h-2}+a_{h-1}+a_h)+d+\al_{h-2}}{2}.
\]
From the maximal choice of $a_2+a_3+\cdots+a_{h}$, we obtain 
\[
a_1=\frac{n+h^2-3h-9+d+\al_{h-2}}{2},
\]
and checking if $a_1<a_2=n-h$ leads to $(\al_{h-2}-n)+(h^2-h-9)+d<0$ which is not compatible with $h\geq 5$. 

This concludes the case $h=j$.

\vspace{3mm}

Let us now address the remaining case $h=j+1$. In this setting we have only three parts in the free area and, as before, the possible choices are the following triple of elements 
\begin{equation}\label{eq:freearea2n-4}
(a_{h-2},a_{h-1},a_h)\in\{(n-4,n-3,n-2),(n-4,n-2,n-1),(n-3,n-2,n),(n-2,n-1,n)\}.
\end{equation}

Let $h=3$. We have $a_1+a_2+a_3=2n-4+d$. In the case $(a_1,a_2,a_3)=(n-2,n-1,n)$ we obtain $d=n+1$, a contradiction. In the other three cases we obtain $d$ equals $n-1$, $n-3$ and $n-5$, or, in other words, $d=n-(2k-1)$ for $1 \leq k \leq 3$ as claimed. The remaining cases $3\leq d\leq n-7$ are considered by showing partitions obtained in the case $h=4$.

Let $h=4$. We have
\[
a_1=\frac{2(2n-4)-(a_2+a_3+a_4)+d}{2}.
\]
Notice that, for each choice of $(a_2,a_3,a_4)$, we have $a_2+a_3+a_4=3n- (2t+1)$, for some $t\geq 0$. Therefore, since $a_1$ is an integer, $n$ is even if and only if $d$ is odd. Precisely, $d=n-(2k-1)$ for some $1\leq k\leq \lfloor(n-2)/2\rfloor$ (recall that, by Proposition~\ref{prop:d1} and Proposition~\ref{prop:d2}, the cases $d=1,2$ are not maximal when $\l_t=2n-4$). To prove that for all $3\leq d\leq n-7$ there exists $\l\in\mup{T_{n,d}}$, consider, for example, the assignment  $(a_2,a_3,a_4)=(n-3,n-2,n)$. In this case, from $\su$, we obtain $a_1=(n-3+d)/2$ which satisfies $a_1<n-4$ if and only if $d<n-5$ and the corresponding partition is unrefinable. Indeed, we can violate the refinability only if either $a_1+a_2=\al_1$ or  $a_1+a_3=\al_1$ or $a_1+a_4=\al_1$, and this is only possible if $d\in\{-1,1,2\}$, a contradiction.

If $h\geq 5$, from $\su$ we have 
\[
a_1+a_2+\cdots+a_{h-3}=\frac{(h-2)(2n-4)-(a_{h-2}+a_{h-1}+a_h)+d}{2},
\]
from which we obtain again that $n$ is even if and only if $d$ is odd, i.e., $d=n-(2k-1)$ for some positive integer $k$.
\end{proof}

\begin{corollary}\label{cor:d=3}
Let $n$ be odd. Then $\#\mup{T_{n,3}}=1$.
\end{corollary}
\begin{proof}
By Proposition~\ref{prop:2n-4dno0}, if $d\not\in\l$ then there does not exist any maximal unrefinable partition with $d=3$. By Proposition~\ref{prop:d3}, if $d\in\l$ then we have that, for $h=j+1$, $n$ odd implies $d$ even and, for $h=j$,  there exists only one maximal unrefinable partition with $d=3$.
\end{proof}

\subsection{The case $\l_t=2n-5$}\label{sec:descr2nminus5}
Also in this case, by Proposition~\ref{prop:maxlt}, we have to consider both cases $d\not\in\l$ and $d\in\l$. Let us address the two cases separately.
\begin{proposition}\label{prop:2n-5dno0}
Let $\listP\in\mup{T_{n,d}}$ such that $d\not\in\l$. If $\l_t=2n-5$, then $d=n-6$ and such a partition is unique.
\end{proposition}
\begin{proof}
We derive the claim from proving the following two statements:
\begin{enumerate}
\item if $d\leq n-6$, then $d=n-6$ and there exists only one partition;
\item no partition exists if $n-5\leq d <n$.
\end{enumerate}
Let us now prove each claim separately.

\begin{enumerate}
\item If $\l_t=2n-5$ the free area is $\{n-5,\ldots,n\}$ and we have, by Remark~\ref{rmk1}, that $h=j+1$. Moreover, from the fact that $d\not\in\l$ and $d\leq n-6$, or in other words $d$ is outside the free area, we must have $\l_t-d\in\l$ since $\l$ is unrefinable. Hence we are left with $j-2=h-3$ parts in the last area to be determined. Now, as already concluded in the case $\l_t=2n-4$, we can only choose three parts in the free area. In particular we have only eight possible cases, i.e.,
\[
    \left(a_{h-2},a_{h-1},a_h\right)=\left\{ 
    \begin{aligned} 
   & \left(n-5,n-4,n-3\right),  \left(n-5,n-4,n-2\right),\\
    &\left(n-5,n-3,n-1\right), \left(n-5,n-2,n-1\right),\\
    &\left(n-4,n-3,n\right),  \left(n-4,n-2,n\right),\\
    &\left(n-3,n-1,n\right),  \left(n-2,n-1,n\right)
    \end{aligned}
    \right\}.
    \]
    
    Let $h=3$. From $\su$ we obtain eight equations by the eight possible options in the free area
\begin{align*}
    &4n-10-d=3n-3,\\
    &4n-10-d=3n-4,\\
    &4n-10-d=3n-6,\\
    &4n-10-d=3n-7,\\
    &4n-10-d=3n-8,\\
    &4n-10-d=3n-9,\\
    &4n-10-d=3n-11,\\
    &4n-10-d=3n-12.
\end{align*}
 In the first case we obtain $d=n-7$ which is a contradiction since  if $\l\in\mup{T_{n,n-7}}$ then $\l_t=2n-4$. In the last six cases we have $d> n-5$ and so we obtain a contradiction. From the second one we obtain $d=n-6$,  corresponding to the partition 
\[
\l=(1,\dots,n-7,n-5,n-4,n-2,n+1,2n-5),
\]
which is unrefinable for $n\geq 11$.

Let $h=4$. As usual, maximizing $a_2+a_3+a_4$, from $\su$ we determine
    \[
    a_1=\frac{3n-12-d}{2}.
    \]
    Imposing $a_1<n-5$ we obtain $n-2-d<0$,  a contradiction.
    
    Let $h\geq 5$. Maximizing $a_2+a_2+\cdots+a_h$, from $\su$ we obtain
    \[
a_1=\frac{3n+h^2-2h-20-d}{2},
\]
being  meaningful when $a_1<a_2=n-h-1$, from which we obtain 
\[
n-d+(h^2-18)<0,
\]
a contradiction if $h\geq 5$.

\item Notice that, since $d\not\in\l$ and $n-5\leq d< n$, we can only choose $a_{h-1}$ and $a_h$ in the free area, being the third spot occupied already by $d$. From $\su$, in this case we have 
\begin{equation}\label{eq:2n-5h5}
a_1+a_2+\cdots+a_{h-2}=\frac{(h-1)(2n-5)-(a_{h-1}+a_h)}{2}.
\end{equation}
We already know that if $d=n-(2k-1)$ and $\l\in\mup{T_{n,d}}$, then $\l_t=2n-4$, so we can suppose that $d\in\{n-4,n-2\}$.

\vspace{1.5mm}

First suppose that $d=n-4$. In this case 
\[
(a_{h-1},a_h)\in\{(n-5,n-3),(n-5,n-2),(n-3,n),(n-2,n)\}.
\]

Let $h=3$. Maximizing $a_2+a_3$, we can calculate
\begin{equation}\label{eq:dn-4h3}
a_1=\frac{4n-10-(a_2+a_3)}{2},
\end{equation}
so $a_2+a_3$ must be an even number. Now if $a_2+a_3=2n-2$, we obtain $a_1=n-4$, a contradiction, and if the sum is $a_2+a_3=2n-8$, we obtain $a_1=n-1$, again a contradiction.

Let $h\geq 4$. Maximizing $a_2+a_3+\cdots+a_h$, from $\su$ we obtain 
\[
a_1=\frac{2n+h^2-17}{2}.
\]
Checking if $a_1<a_2=n-h-2$, we derive that $h^2+h-13<0$, a contradiction if $h\geq 4$. 

\vspace{1.5mm}

Now suppose that $d=n-2$. In this case 
\[
(a_{h-1},a_h)\in\{(n-5,n-4),(n-5,n-1),(n-4,n),(n-1,n)\}.
\]

Let $h=3$. Maximizing $a_2+a_3$, from Eq.~\eqref{eq:dn-4h3}
 $a_2+a_3$ must be an even number. Now if $a_2+a_3=2n-4$, we obtain $a_1=n-3$, a contradiction, and if the sum is $a_2+a_3=2n-6$, we obtain $a_1=n-2$, again a contradiction.

Let $h\geq 4$. Maximizing $a_2+a_3+\cdots+a_h$, from $\su$ we obtain 
\[
a_1=\frac{2n+h^2-18}{2}.
\]
Checking if $a_1<a_2=n-h-2$, we derive that $h^2+h-14<0$, a contradiction if  $h\geq 4$.\qedhere

\end{enumerate}
\end{proof}

Let us address the case $d\in\l$. Recall that, in this case, by Remark~\ref{rmk1} we have $h\in\{j,j+1\}$.
\begin{proposition}\label{prop:2n-5dno0_2}
Let $\listP\in\mup{T_{n,d}}$ be such that $\l_t=2n-5$ and $d\in\l$. If $h=j$, then $d=4$, $n$ is even and such a partition is unique. If $h=j+1$, then for each $1\leq k\leq \lfloor(n-4)/2\rfloor$ there exists $\l\in\mup{T_{n,d}}$ with $d=n-2k$.
\end{proposition}
    
\begin{proof}
Let us assume that $h=j$. Since, by Remark~\ref{rmk2}, $\al_{h-1}=d$ and $\al_{h}=2n-5$ are already fixed, then $a_{h-1}$ and $a_h$ are parts of the free area. Notice that the free area cannot contain more than three parts. 

\medskip

Let us first assume that it only contains two parts, i.e., 

\[
    \left(a_{h-1},a_h\right)\in\left\{ 
    \begin{aligned} 
   & \left(n-5,n-4\right),  \left(n-5,n-3\right),\\
    &\left(n-5,n-2\right), \left(n-5,n-1\right),\\
    &\left(n-4,n-3\right),  \left(n-4,n-2\right),\\
    &\left(n-4,n\right),  \left(n-3,n-1\right),\\
    &(n-3,n),(n-2,n-1),\\
    &(n-2,n),(n-1,n)
    \end{aligned}
    \right\}.
    \]
    
    Let $h=2$. We must have $a_1+a_2=2n-5+d$. In the case when $(a_1,a_2)=(n-1,n)$ we obtain $d=4$. which is the claim, since the corresponding partition $\l=(1,2,\ldots,n-2,2n-5)$ is unrefinable. In all the other cases we obtain either $0<d<4$ which, by Proposition~\ref{prop:d1}, Proposition~\ref{prop:d2} and Proposition~\ref{prop:d3}, contradicts the maximality of $\l$, or $d\leq 0$ which is also a contradiction.

Let $h=3$. From $\su$, we have
\[
a_1=\frac{4n-10-(a_2+a_3)+d}{2}.
\]
Considering the maximal choice $(a_2,a_3)=(n-1,n)$ we obtain $a_1=(2n-9+d)/2<n-5$ when $d< -1$, a contradiction.

Let $h\geq 4$. From $\su$ we obtain
\[
a_1+a_2+\cdots+a_{h-2}=\frac{(h-1)(2n-5)-(a_{h-1}+a_h)+d}{2}.
\]
Maximizing $a_2+\ldots+a_{h}$, we have
\[
a_1=\frac{2n+h^2-h-18+d}{2},
\]
which satisfies $a_1<a_2=n-h-2$ when $h^2+h-14+d<0$, a contradiction when $h\geq 4$. 

Under the assumption that $h=j$, it remains to consider the case of three parts in the free area, i.e.,
\begin{equation}\label{eq:freearea2n-5}
  \left(a_{h-2},a_{h-1},a_h\right)\in\left\{ 
    \begin{aligned} 
   & \left(n-5,n-4,n-3\right),  \left(n-5,n-4,n-2\right),\\
    &\left(n-5,n-3,n-1\right), \left(n-5,n-2,n-1\right),\\
    &\left(n-4,n-3,n\right),  \left(n-4,n-2,n\right),\\
    &\left(n-3,n-1,n\right),  \left(n-2,n-1,n\right)
    \end{aligned}
    \right\}.
\end{equation}
In this case $a_1,a_2,\ldots,a_{h-3}$ determine $\al_1,\al_2,\ldots,\al_{h-3=j-3}$, while the element  $n+1 \leq \al_{j-2} \leq 2n-6$ is not determined by one of the $a_i$s.

Let $h=3$. We have $a_1+a_2+a_3=2n-5+d+\al_1$, with $n+1\leq \al_1\leq 2n-5$.  If $a_1=n-2$, we obtain $d+\al_1=n+2$ and so we have the only possibility of $d=1$ and $\al_1=n+1$, which, by Proposition~\ref{prop:d1}, contradicts the maximality of $\l$. In all the other cases we obtain $d<0$, a contradiction.

Let $h=4$. From $\su$ we have 
\[
a_1=\frac{4n-10-(a_2+a_3+a_4)+d+\al_2}{2},
\]
and maximizing $a_2+a_3+a_4$ we obtain $a_1=(n-7+d+\al_2)/2$. Since $a_1<n-5$, then $\al_2 < n-3-d < n-3$, a contradiction.

Let $h\geq 5$.  From $\su$ we have
\[
a_1+a_2+\cdots+a_{h-3}=\frac{(h-2)(2n-5)-(a_{h-2}+a_{h-1}+a_h)+d+\al_{h-2}}{2}.
\]
From the maximal choice of $a_2+a_3+\cdots+a_{h}$, we obtain 
\[
a_1=\frac{n+h^2-2h-15+d+\al_{h-2}}{2},
\]
and checking if $a_1<a_2=n-h-1$ leads to $(\al_{h-2}-n)+(h^2-13)+d<0$ which is not compatible with $h\geq 5$. 

This concludes the case $h=j$.

\medskip

Let us now address the remaining case $h=j+1$. In this setting we have only three parts in the free area and the possible choices are those in the set presented in Eq.~\eqref{eq:freearea2n-5}.

Let $h=3$. We have $a_1+a_2+a_3=2n-5+d$. In the case $(a_1,a_2,a_3)\in\{(n-3,n-1,n),(n-2,n-1,n)\}$ we obtain $d>n$, a contradiction. If $(a_1,a_2,a_3)=(n-4,n-2,n)$ or $(n-5,n-2,n-1)$ or $(n-5,n-4,n-3)$ then we obtain respectively $d=n-1$, $n-3$ and $n-7$ which, by Proposition~\ref{prop:d3}, contradict the maximality of $\l$. In the other three cases we obtain $d$ equals $n-2$, $n-4$ and $n-6$, or, in other words, $d=n-2k$ for $1 \leq k \leq 3$ as claimed. The remaining cases $4\leq d\leq n-8$ are considered by showing partitions obtained in the case $h=4$.

Let $h=4$. We have
\[
a_1=\frac{2(2n-5)-(a_2+a_3+a_4)+d}{2}.
\]
Notice that, for each choice of $(a_2,a_3,a_4)$ in $\{(n-2, n - 1, n), (n - 4, n - 3, n), (n - 5, n - 3, n - 1), (n - 5, n - 4, n - 2)\}$, we have that $n$ is even if and only if $d$ is odd, which, by Proposition~\ref{prop:d3}, contradicts the maximality of $\l$. In the other four cases, since $a_1$ is an integer, we obtain that $d=n-2k$, for some integer $k$.
 To prove that for all $4\leq d\leq n-8$ there exists $\l\in\mup{T_{n,d}}$, consider for example, the assignment  $(a_2,a_3,a_4)=(n-4,n-2,n)$. In this case, from $\su$, we obtain $a_1=(n-4+d)/2$ which satisfies $a_1<n-5$ if and only if $d<n-6$ and the corresponding partition is unrefinable. Indeed, we can violate the refinability only if either $a_1+a_2=\al_1$ or  $a_1+a_3=\al_1$ or $a_1+a_4=\al_1$, and this is only possible if $d\in\{-1,1,3\}$, a contradiction by Proposition~\ref{prop:d1}, Proposition~\ref{prop:d3} and  since $d>0$.
 
 If $h\geq 5$, from $\su$ we have 
\[
a_1+a_2+\cdots+a_{h-3}=\frac{(h-2)(2n-5)-(a_{h-2}+a_{h-1}+a_h)+d}{2}.
\]
Since $2n-5$ is odd, we have to consider the parity of $h-2$, and so of $h$. If $h$ is even, then we obtain a contradiction for $(a_{h-2},a_{h-1},a_h)$ in $\{(n-2, n - 1, n), (n - 4, n - 3, n), (n - 5, n - 3, n - 1), (n - 5, n - 4, n - 2)\}$ and $d=n-2k$ for some positive integer $k$ in the other cases. Instead if $h$ is odd we obtain a contradiction for $(a_{h-2},a_{h-1},a_h)$ in $\{(n - 3, n - 1, n), (n - 4, n - 2, n)), (n - 5, n - 2, n - 1)), (n - 5, n - 4, n - 3))\}$ and again $d=n-2k$ for some positive integer $k$ in the other cases.
\end{proof}

\section{Counting maximal unrefinable partitions}\label{sec:count}
In the previous section we proved the existence of maximal unrefinable partitions with specific parameters. We use those results in the current section to specify all the possible configurations meeting the requirements and therefore counting the corresponding number of partitions.
The two cases to be considered are addressed in this section using the same strategy. Therefore, despite the problems have a slightly different combinatorial structure, we try to use a similar notation and terminology in Sec.~\ref{sec:2nminus4} and in Sec.~\ref{sec:2nminus5}.

\subsection{The case $\l_t=2n-4$}\label{sec:2nminus4} We have already proved that if $\l\in\mup{T_{n,d}}$, then $d=3$ or $d=n-(2k-1)$, for $1\leq k \leq\lfloor(n-2)/2\rfloor$. We denote by $(\al_1,\ldots,\al_j)\setminus (a_1,\ldots,a_h)$ the partition $\l$ obtained from $\pi_{n,d}$ by removing the elements $a_i$s and replacing them with the elements $\al_i$s. We have already shown in the previous section that, when $h\leq 3$, only the following partitions belong to $\mup{T_{n,d}}$ (cf. the proof of Proposition~\ref{prop:d3}, computing $\al_i$ from the corresponding $a_i$):
\begin{itemize}
\item $(3,2n-4)\setminus (n-1,n)$ for $d=3$,
\item $(n-1,2n-4)\setminus (n-3,n-2,n)$ for $d=n-1$,
\item $(n-3,2n-4)\setminus (n-4,n-2,n-1)$ for $d=n-3$,
\item $(n-5,2n-4)\setminus (n-4,n-3,n-2)$ for $d=n-5$,
\item $(n+1,2n-4)\setminus (n-2,n-1,n)$ for $d=n-5$ (cf. Proposition~\ref{prop:2n-4dno0}).
\end{itemize}

Recall that for $h\geq 4$ we have the following choices for the free area, $(a_{h-2},a_{h-1},a_h)\in\{(n-4,n-3,n-2),(n-4,n-2,n-1),(n-3,n-2,n),(n-2,n-1,n)\}$. From now on, according to Proposition~\ref{prop:2n-4dno0}, we must only consider the case $d\in\l$. Indeed, the only maximal unrefinable partition with $d\not\in\l$ and $\l_t=2n-4$ is the fifth partition in the previous list.

Let $h=4$. We have
\[
a_1=\frac{2(2n-4)-(a_2+a_3+a_4)+d}{2},
\]
and, assigning all the possible values to $a_2,a_3,a_4$, we obtain the partitions:
\begin{itemize}
   \item $(d,(3n-3-d)/2,2n-4)\setminus ((n-5+d)/2,n-2,n-1,n)$, with $d<n-5$, otherwise $a_1\geq a_2$, and $d\ne 3$, otherwise $a_1+a_2=\al_2$;
    \item $(d,(3n-5-d)/2,2n-4)\setminus((n-3+d)/2,n-3,n-2,n)$, with $d<n-5$;
    \item $(d,(3n-7-d)/2,2n-4)\setminus((n-1+d)/2,n-4,n-2,n-1)$, with $d<n-7$;
    \item $(d,(3n-9-d)/2,2n-4)\setminus ((n+1+d)/2,n-4,n-3,n-2)$, with $d<n-9$. 
\end{itemize}

Let $h\geq 5$. We have
\[
a_1=\frac{n+(h^2-3h-9)+d}{2}
\]
obtained from the maximal choice for $a_2+a_3+\cdots +a_{h-3}$ and from $a_{h-2}=n-2$, $a_{h-1}=n-1$ and $a_h=n$, which is also the maximal choice in the free area, and we obtain the partition
\begin{align*}
&\left(d,n+1,\dots,n+h-4,\frac{3n-(h^2-3h-1)-d}{2},2n-4\right)  \setminus \\ 
 &   \left(\frac{n+(h^2-3h-9)+d}{2},n-h,\dots,n-5,n-2,n-1,n\right),
\end{align*}
with $d<n-(h^2-h-9)$.
Notice that $a_1+a_2 > \al_1$, therefore the obtained partition is unrefinable.

All the others, obtained for the remaining possibilities for $a_2+a_3+\cdots +a_{h-3}$, are obtained by replacing $(a_1,a_2,\ldots,a_{h-3})$ with $(a_1+i,a_2-i_1,\ldots,a_{h-3}-i_{h-4})$, where $i=\sum_{r=1}^{h-4}i_r$ and such that $a_1+i < a_2-i_1 < \dots < a_{h-3}-i_{h-4}$.

We proceed similarly for the other three choices in the free area. All the results are summarized in Table~\ref{table1} (displayed at the end of the paper). The first row of the table is Corollary~\ref{cor:d=3}  and the next four rows are summarized in the following three results.
\begin{corollary}
$\#\mup{T_{n,n-1}}=1.$
\end{corollary}
\begin{corollary}
$\#\mup{T_{n,n-3}}=1.$
\end{corollary}
\begin{corollary}
$\#\mup{T_{n,n-5}}=2.$
\end{corollary}

We are now ready to address the remaining cases, i.e., to  compute explicitly the  number of partitions $\#\mup{T_{n,d}}$ when $3\leq d \leq n-7$ and $d=n-(2k-1)$. Notice that, by  Proposition~\ref{prop:2n-4dno0} and Proposition~\ref{prop:d3}, we know that $d\in\l$ and that the partition is uniquely determined when we are given $n$, $d$ and the elements $a_1,a_2,\ldots,a_h$ to be removed. Moreover, from Eq.~\eqref{eq:freearea2n-4} we have four possible choices for the three elements in the free area which are symmetric with respect to $n-2$, therefore the partitions are determined by the list of the $a_i$s which are smaller than or equal to $n-3$. Only one partition is exceptional with respect to this representation, i.e., the partition 
\[
\pi \deq (3,2n-4)\setminus (n-1,n)= (1,2,\dots, n-2,2n-4).
\]

\begin{definition}\label{ustar2nminus4}
Let  $d=n-(2k-1)$ with $3\leq d \leq n-7$.
Let us define the set of missing parts, for each $\l\in\mup{T_{n,d}}$, which are  smaller than or equal to $n-3$:
\[
\mup{T_{n,d}}^*\deq\{\eta=(\eta_1,\eta_2,\ldots,\eta_s) \mid s \geq 0, \eta_i\in\Ml, \l\in\mup{T_{n,d}}, \eta_i\leq n-3\}.
\]
Notice that $\pi$ corresponds to the empty partition $()\in \mup{T_{n,3}}^*$ obtained for $s=0$.
\end{definition}
From the previous argument,  $\mup{T_{n,d}}$ is in one-to-one correspondence with $\mup{T_{n,d}}^*$. In order to prove the claimed bijection, let us  introduce a partition of the set $\mup{T_{n,d}}^*$ which is convenient for our  purposes.
 
 \begin{definition}\label{def:partitionofU*}
 Let $n$, $d$ and $h$ be positive integers. Let us define
\begin{align*}
\mA_{n,d,h}\deq&\left\{\eta\in\mup{T_{n,d}}^*\mid|\eta|=h-3, \eta_{h-3}\leq n-5\right\},\\
\mB_{n,d,h}\deq&\left\{\eta\in\mup{T_{n,d}}^*\mid|\eta|=h-2, n_{h-3}\leq n-5, \eta_{h-2}=n-3 \right\},\\
\mC_{n,d,h}\deq&\left\{\eta\in\mup{T_{n,d}}^*\mid|\eta|=h-2, \eta_{h-3}\leq n-5, \eta_{h-2}=n-4\right\},\\
\mD_{n,d,h}\deq&\left\{\eta\in\mup{T_{n,d}}^*\mid|\eta|=h-1, \eta_{h-2}=n-4, \eta_{h-1}=n-3\right\},
\end{align*}
and
\begin{align*}
\mA_{n,d}\deq&\bigcup_{h\geq4}\mA_{n,d,h}, &\mB_{n,d}\deq&\bigcup_{h\geq4}\mB_{n,d,h},\\
\mC_{n,d}\deq&\bigcup_{h\geq4}\mC_{n,d,h}, &\mD_{n,d}\deq&\bigcup_{h\geq4}\mD_{n,d,h}.
\end{align*}
Reading Table~\ref{table1}, we can note that  
\begin{equation}\label{eq:ud4d3}
\mup{T_{n,d}}^*=
\begin{cases}
\mA_{n,d}\,\dot\cup\,\mB_{n,d}\,\dot\cup\,\mC_{n,d}\,\dot\cup\,\mD_{n,d} &\quad 4 \leq d \leq n-7,\\
\{()\}  \,\dot\cup\, \bigcup_{h\geq5}\mA_{n,d,h} \,\dot\cup\,  \mB_{n,d}\,\dot\cup\,\mC_{n,d}\,\dot\cup\,\mD_{n,d} &\quad d=3.\\

\end{cases}
\end{equation}
 \end{definition}

 Analogously, let us now  introduce a  convenient partition of $\dist_{r}$, that we will prove to be related with that of Definition~\ref{def:partitionofU*}.
\begin{definition}\label{def:partDr}
Let $r$ and $s$ be  positive integers. Let us define
\begin{align*}
\mA_{r,s}^*\deq&\left\{\l\in\mathbb{D}_{r,s}\mid \l_1\geq3\right\},\\
\mB_{r,s}^*\deq&\left\{\l\in\mathbb{D}_{r,s}\mid\l_1=1, \l_2\geq 3\right\},\\
\mC_{r,s}^*\deq&\left\{\l\in\mathbb{D}_{r,s}\mid \l_1=2\right\},\\
\mD_{r,s}^*\deq&\left\{\l\in\mathbb{D}_{r,s}\mid \l_1=1, \l_2=2\right\}.
\end{align*}
It is clear that
\begin{equation}\label{eq:dist}
\mathbb{D}_r=\bigcup_{s\geq2}\mA_{r,s}^*\dot\cup\bigcup_{s\geq2}\mB_{r,s}^*\dot\cup\bigcup_{s\geq2}\mC_{r,s}^*\dot\cup\bigcup_{s\geq3}\mD_{r,s}^*.
\end{equation}
\end{definition}

Finally, let us define the following correspondence from $\mup{T_{n,d}}^*\setminus \left\{ \mA_{n,d,4} \cup\{() \}\right\}$ to $\dist$. We will discuss later how to extend the values of the function on the partitions of $\left\{ \mA_{n,d,4} \cup\{() \}\right\}$. 
\begin{definition}\label{def:phi}
Let us denote
\[
\begin{array}{rrcl}
\phi:&\mup{T_{n,d}}^* \setminus \left\{ \mA_{n,d,4} \cup\{() \}\right\} & \longrightarrow & \dist\\
&(\eta_1,\eta_2,\ldots,\eta_t)&\longmapsto&(n-2-\eta_t,\ldots,n-2-\eta_2,n-2-\eta_1).
\end{array}
\]
\end{definition}

The two following propositions will be used in Theorem~\ref{thm:maindodd} to prove the part of our main statement, introduced in Sec.~\ref{sec:contri}, concerning the case $\l_t=2n-4$.
\begin{proposition}\label{prop:dispA}
Let $d=n-(2k-1)$ such that $3\leq d\leq n-7$ and $h \geq 5$. Then $\phi$ sends bijectively $\mA_{n,d,h}$ into $\mA_{k,h-3}^*$.
\end{proposition}
\begin{proof}
Let us start by proving that the correspondence is well defined, i.e., if $\eta\in \mA_{n,n-(2k-1),h}$, then $\phi(\eta)\in \mA_{k,h-3}^*$. 
Let $\eta \in \mA_{n,n-(2k-1),h}$. Then, by Table~\ref{table1}, 
\[
\eta=\left(\frac{n+\left(h^2-3h-9\right)+d}{2}+i,n-h-i_1,\dots,n-5-i_{h-4}\right),
\]
for some positive integers $i,i_1\geq i_2\geq\ldots \geq i_{h-4}$ such that $i=\sum_{j=1}^{h-4}i_j$. By definition of $\phi$ we have
\[
\phi(\eta)=\left(3+i_{h-4},4+i_{h-5},\dots,h-2+i_1,k+\frac{-h^2+3h+4}{2}-i\right).
\]
Notice that $|\phi(\eta)|=h-3$, $\phi(\eta)_1\geq 3$ and that $\phi(\eta)\vdash k$. Therefore $\phi(\eta)\in\mA_{k,h-3}^*$.

Notice also that $\phi$ is trivially injective and so, in order to conclude the proof, it remains to prove that $\phi$ is surjective from $\mA_{n,n-(2k-1),h}$ to $\mA_{k,h-3}^*$. For this purpose, let $\rho=(\rho_1,\rho_2,\ldots,\rho_{h-3})\in\mA_{k,h-3}^*$. Then the general expression for such $\rho$ is 
\[
\rho=\left(3+i_{1},4+i_{2},\dots,h-2+i_{h-4},k+\frac{-h^2+3h+4}{2}-i\right),
\] 
for some positive integers $i,i_1\leq i_2 \leq \ldots \leq i_{h-4}$ such that $i=\sum_{j=1}^{h-4}i_j$. It is easy to see that
\[
\eta\deq(n-2-\rho_{h-3},\ldots,n-2-\rho_2,n-2-\rho_1)
\]
is such that $\phi(\eta)=\rho$. We need to prove that $\eta\in\mA_{n,n-(2k-1),h}$. We have 
\begin{align*}
\eta=&\left(n-2-\left(k+\frac{-h^2+3h+4}{2}-i\right),n-2-(h-2+i_{h-4}),\ldots,n-2-(3+i_1)\right)\\
=&\left(\frac{n+d+h^2-3h-9}{2}+1,n-h-i_{h-4},\ldots,n-5-i_1\right),
\end{align*}
which, from Table~\ref{table1}, is exactly the generic form of a partition in $\mA_{n,n-(2k-1),h}$.
\end{proof}
Similar computations lead to the corresponding results for $\mB^*$, $\mC^*$ and $\mD^*$. Precisely:
\begin{proposition}\label{prop:dispBCD}
Let $d=n-(2k-1)$ such that $3\leq d\leq n-7$ and $h \geq 5$. Then $\phi$ sends bijectively 
\begin{enumerate}
\item $\mB_{n,n-(2k-1),h}$ into $\mB_{k,h-2}^*$,
\item $\mC_{n,n-(2k-1),h}$ into $\mC_{k,h-2}^*$, 
\item $\mD_{n,n-(2k-1),h}$ into $\mD_{k,h-1}^*$.
\end{enumerate}
Moreover, in the case $h=4$ we have
\begin{enumerate}
\item $\mB_{n,n-\left(2k-1\right),4}\overset{\phi}{\longleftrightarrow}\{(1,k-1)\}=\mB_{k,2}^*,$
\item $\mC_{n,n-\left(2k-1\right),4}\overset{\phi}{\longleftrightarrow}\{(2,k-2)\}=\mC_{k,2}^*,$
\item $\mD_{n,n-\left(2k-1\right),4}\overset{\phi}{\longleftrightarrow}\{(1,2,k-3)\}=\mD_{k,3}^*.$
\end{enumerate}
\end{proposition}
\begin{remark}\label{rmk:disp4}
Notice that $\mA_{n,d,4}$ contains only the (trivial) partition $\tilde{\eta}\deq (\frac{n-5+d}{2})$, and so, extending the function $\phi$ on  $\mA_{n,d,4}$ implies to consider the (trivial) partition of $k$ in a single part, i.e., $\phi(\tilde{\eta}) = (k)$. Similarly, with an abuse of notation we can assume $() \overset{\phi}{\mapsto} ()$.  This extends the definition  of $\phi$ also on $\left\{ \mA_{n,d,4} \cup\{() \}\right\}$, making the function defined on the whole set $\mup{T_{n,d}}^*$.
\end{remark}
We can finally summarize the above results. From Proposition~\ref{prop:dispA}, Proposition~\ref{prop:dispBCD}, Remark~\ref{rmk:disp4} and from Eq.~\eqref{eq:dist} we obtain:
\begin{theorem}\label{thm:maindodd}
Let $d=n-(2k-1)$ be such that $3\leq d\leq n-7$. Then
\[
\#\mup{T_{n,n-(2k-1)}}=1+\#\dist_k.
\]
\end{theorem}
\begin{proof}
Let us assume first that $4\leq d\leq n-7$. Then, since the empty partition $()$ appears only in the case $d=3$ which is not considered, we have (cf.\ Eq.~\eqref{eq:ud4d3})
\begin{align*}
\mup{T_{n,d}} \leftrightarrow \mup{T_{n,d}}^*
&=\left(\bigcup_{h\geq4}\mA_{n,d,h}\right)\cup\left(\bigcup_{h\geq4}\mB_{n,d,h}\right)\cup\left(\bigcup_{h\geq4}\mC_{n,d,h}\right)\cup\left(\bigcup_{h\geq4}\mD_{n,d,h}\right)\\
&\overset{\phi}{\leftrightarrow} \{(k)\} \cup \left(\bigcup_{s\geq2}\mA_{k,s}^*\right)\cup\left(\bigcup_{s\geq2}\mB_{k,s}^*\right)\cup\left(\bigcup_{s\geq2}\mC_{k,s}^*\right)\cup\left(\bigcup_{s\geq3}\mD_{k,s}^*\right)\\
& = \{(k)\} \cup \mathbb D_k,
\end{align*}
from which we obtained the desired claim.
In the remaining case $d=3$, we proceed in the same way and, using the corresponding description of Eq.~\eqref{eq:ud4d3}, we obtain 
\[
\mup{T_{n,d}} \leftrightarrow \mup{T_{n,d}}^* \overset{\phi}{\leftrightarrow} \{()\} \cup \mathbb D_k.
\]
\end{proof}

\subsection{The case $\l_t=2n-5$}\label{sec:2nminus5} 
We can now count the number of maximal unrefinable partitions in the case of $\l_t=2n-5$. Using the same notation as in Sec.~\ref{sec:2nminus4}, we proceed with a similar argument, despite the increased combinatorial complexity of the problem. We have already proved that if $\l\in\mup{T_{n,d}}$, then $d=n-2k$, for $1\leq k \leq\lfloor(n-4)/2\rfloor$. Moreover,
we have already proved in Sec.~\ref{sec:descr2nminus5} that, when $h\leq 3$, only the following partitions belong to $\mup{T_{n,d}}$ (cf. the proof of Proposition~\ref{prop:2n-5dno0_2}):
\begin{itemize}
\item $\left(4,2n-5\right) \setminus \left(n-1,n\right)$ for $d=4$,
\item $\left(n-2,2n-5\right) \setminus \left(n-4,n-3,n\right)$ for $d=n-2$,
\item $\left(n-4,2n-5\right) \setminus \left(n-5,n-3,n-1\right)$ for $d=n-4$,
\item $\left(n-6,2n-5\right) \setminus \left(n-5,n-4,n-2\right)$ for $d=n-6$,
\item $\left(n+1,2n-5\right) \setminus \left(n-3,n-1,n\right)$ for $d=n-6$ (cf. Proposition~\ref{prop:2n-5dno0}).
\end{itemize}
For $h\geq 4$ we have the following eight choices for the free area, i.e.,
\begin{equation*}
  \left(a_{h-2},a_{h-1},a_h\right)\in\left\{ 
    \begin{aligned} 
   & \left(n-5,n-4,n-3\right),  \left(n-5,n-4,n-2\right),\\
    &\left(n-5,n-3,n-1\right), \left(n-5,n-2,n-1\right),\\
    &\left(n-4,n-3,n\right),  \left(n-4,n-2,n\right),\\
    &\left(n-3,n-1,n\right),  \left(n-2,n-1,n\right)
    \end{aligned}
    \right\}.
\end{equation*}
Also in this case, we only consider the case $d \in \l$ (cf. Proposition~\ref{prop:2n-5dno0}), indeed the only maximal unrefinable partition with $\l_t=2n-5$, obtained assuming $d \notin \l$ is the last of the previous list.

Let $h=4$. Since $n$ is even if and only if $d$ is even, and since, from $\su$, we can calculate
\[
a_1=\frac{2\left(2n-5\right)+d-\left(a_2+a_3+a_4\right)}{2}, 
\]
then $a_2+a_3+a_4$ is even if and only if $n$ is even. Therefore, the only possible choices compliant with the previous requirement are
\[
\left(a_2,a_3,a_4\right)=\left\{ (n-5,n-4,n-3), (n-5,n-2,n-1), (n-4,n-2,n), (n-3,n-1,n) \right\}.
\]
We obtain the partitions
\begin{itemize}
\item $\left(d,{(3n-4-d)}/{2},2n-5\right) \setminus \left({(n-6+d)}/{2},n-3,n-1,n\right)$, with $d<n-6$, otherwise $a_1\geq a_2$, and $d\ne 4$, otherwise $a_1+a_2=\al_2$;
\item $\left(d,{(3n-6-d)}/{2},2n-5\right) \setminus \left({(n-4+d)}/{2},n-4,n-2,n\right)$, with $d<n-4$;
\item $\left(d,{(3n-8-d)}/{2},2n-5\right) \setminus \left({(n-2+d)}/{2},n-5,n-2,n-1\right)$, with $d<n-8$;
\item $\left(d,{(3n-12-d)}/{2},2n-5\right) \setminus \left({(n+2+d)}/{2},n-5,n-4,n-3\right)$, with $d<n-12$.
\end{itemize} 

If $h\geq 5$, we need to distinguish the two cases $h$ odd and $h$ even, as already observed at the end of Proposition~\ref{prop:2n-5dno0_2}. The only difference between the two cases  is in the triple $(a_{h-2},a_{h-1},a_{h})$ to be chosen in the free area. Let $h\geq 5$, $h$ odd. We have
\begin{equation*}
a_1=\frac{n+\left(h^2-2h-15\right)+d}{2}
\end{equation*}
obtained from the maximal choice for $a_2+a_3+\cdots +a_{h-3}$ and from $a_{h-2}=n-2$, $a_{h-1}=n-1$ and $a_h=n$, which is also the maximal choice in the free area, and we obtain the partition
\begin{align*}
&\left(d,n+1,\dots,n+h-4,\frac{3n-\left(h^2-2h-5\right)-d}{2}\right)\setminus\\
& \left(\frac{n+\left(h^2-2h-15\right)+d}{2},n-h-1,\dots,n-6,n-2,n-1,n\right)
\end{align*}
with $d\leq n-\left(h^2-11\right)$.
Notice that $a_1+a_2 > \al_1$, therefore the obtained partition is unrefinable.
The remaining cases for $h$ are treated analogously.

All the other partitions, obtained for the remaining possibilities for $a_2+a_3+\cdots +a_{h-3}$, are obtained by replacing $(a_1,a_2,\ldots,a_{h-3})$ with $(a_1+i,a_2-i_1,\ldots,a_{h-3}-i_{h-4})$, where $i=\sum_{r=1}^{h-4}i_r$ and such that $a_1+i < a_2-i_1 < \dots < a_{h-3}-i_{h-4}$.

We proceed similarly for the other seven choices in the free area. All the results are summarized in Table~\ref{table2} (displayed at the end of the paper), and the following consequences are easily noted.

\begin{corollary}
$\#\mup{T_{n,n-2}}=1$.
\end{corollary}

\begin{corollary}
$\#\mup{T_{n,n-4}}=1$.
\end{corollary}

\begin{corollary}
$\#\mup{T_{n,n-6}}=2$.
\end{corollary}

As in the previous section, 
it remains to compute $\#\mup{T_{n,d}}$ when $4\leq d \leq n-8$ and $d=n-2k$. Notice that the partition is uniquely determined when we are given $n$, $d$ and the list of the $a_i$s which are smaller than or equal to $n-3$. Only one partition is exceptional with respect to this representation, i.e., the partition 
\[
\tau \deq (4,2n-5)\setminus (n-1,n)= (1,2,\dots, n-2,2n-5).
\]

The following definition is the counterpart of Definition~\ref{ustar2nminus4} for the case under consideration here.
The defined set will be again in one-to-one correspondence with $\mup{T_{n,d}}$.

\begin{definition}\label{ustar2nminus5}
Let  $d=n-2k$ with $4\leq d \leq n-8$.
Let us define the set of missing parts, for each $\l\in\mup{T_{n,d}}$, which are  smaller than or equal to $n-3$:
\[
\mup{T_{n,d}}^*\deq\{(\eta_1,\eta_2,\ldots,\eta_s) \mid s \geq 0, \eta_i\in\Ml, \l\in\mup{T_{n,d}}, \eta_i\leq n-3\}.
\]
Notice that $\tau$ corresponds to the empty partition $()\in \mup{T_{n,4}}^*$ obtained for $s=0$.
\end{definition}

Notice that when $h \geq 5$ we have (cf. Table~\ref{table2})
\[a_1=\frac{n+d+\left(h^2-2h+t\right)}{2}\] for some  $t\in\mathbb{Z}$.
Since the numerator must be even, we have that $h$ is even if and only if $t$ is even. From this,  we obtain a convenient partition of the set $\mup{T_{n,d}}^*$, similar 
to that introduced in Sec.~\ref{sec:2nminus4}, but which takes into account also the parity of $h$. 
%\section{Acknowledgment}

\begin{definition}
Let $n$, $d$ and $h$ be positive integers.
If $h$ is odd, let us define
\[
\begin{aligned}
&\mE_{n,d,h}^1\deq\left\{\eta\in\mup{T_{n,d}}^*\mid|\eta|=h-3, \eta_{h-3}\leq n-6\right\},\\
&\mE_{n,d,h}^2\deq\left\{\eta\in\mup{T_{n,d}}^*\mid|\eta|=h-1, \eta_{h-3}\leq n-6, \eta_{h-2}=n-4, \eta_{h-1}=n-3\right\},\\
&\mE_{n,d,h}^3\deq\left\{\eta\in\mup{T_{n,d}}^*\mid|\eta|=h-1, \eta_{h-2}=n-5, \eta_{h-1}=n-3\right\},\\
&\mE_{n,d,h}^4\deq\left\{\eta\in\mup{T_{n,d}}^*\mid|\eta|=h-1, \eta_{h-2}=n-5, \eta_{h-1}=n-4\right\}.
\end{aligned}
\]
Moreover
\[
\begin{aligned}
&\mE_{n,d}^1\deq\bigcup_{h\geq5}\mE_{n,d,h}^1, &\mE_{n,d}^2\deq\bigcup_{h\geq5}\mE_{n,d,h}^2,\\
&\mE_{n,d}^3\deq\bigcup_{h\geq5}\mE_{n,d,h}^3, &\mE_{n,d}^4\deq\bigcup_{h\geq5}\mE_{n,d,h}^4.
\end{aligned}
\]

If $h$ is even, let us define
\[
\begin{aligned}
&\mF_{n,d,h}^1\deq\left\{\eta\in\mup{T_{n,d}}^*\mid|\eta|=h-2, \eta_{h-3}\leq n-6, \eta_{h-2}=n-3\right\},\\
&\mF_{n,d,h}^2\deq\left\{\eta\in\mup{T_{n,d}}^*\mid|\eta|=h-2, \eta_{h-3}\leq n-6, \eta_{h-2}=n-4\right\},\\
&\mF_{n,d,h}^3\deq\left\{\eta\in\mup{T_{n,d}}^*\mid|\eta|=h-3, \eta_{h-3}\leq n-6, \eta_{h-2}=n-5\right\},\\
&\mF_{n,d,h}^4\deq\left\{\eta\in\mup{T_{n,d}}^*\mid|\eta|=h, \eta_{h-2}=n-5, \eta_{h-1}=n-4, \eta_{h}=n-3\right\}.
\end{aligned}
\]
Moreover
\[
\begin{aligned}
&\mF_{n,d}^1\deq\bigcup_{h\geq4}\mF_{n,d,h}^1, &\mF_{n,d}^2\deq\bigcup_{h\geq4}\mF_{n,d,h}^2,\\
&\mF_{n,d}^3\deq\bigcup_{h\geq4}\mF_{n,d,h}^3, &\mF_{n,d}^4\deq\bigcup_{h\geq4}\mF_{n,d,h}^4.
\end{aligned}
\]
Finally, let us denote
\[
\begin{aligned}
&\mE_{n,d}\deq\mE_{n,d}^1\cup\mE_{n,d}^2\cup\mE_{n,d}^3\cup\mE_{n,d}^4,\\
&\mF_{n,d}\deq\mF_{n,d}^1\cup\mF_{n,d}^2\cup\mF_{n,d}^3\cup\mF_{n,d}^4.
\end{aligned}
\]
\end{definition}
Reading Table~\ref{table2}, we can note that  
\begin{equation}\label{eq:ud4}
\mup{T_{n,d}}^*=\begin{cases}
 \mE_{n,d}\,\dot\cup\,\mF_{n,d} & 5\leq d\leq n-8,\\  
\mE_{n,d}\,\dot\cup\,\left\{\left(\right)\right\}\,\dot\cup\,\left(\bigcup_{h\geq6}\mF_{n,d,h}^1\right)\,\dot\cup\,\mF_{n,d}^2\,\dot\cup\,\mF_{n,d}^3\,\dot\cup\,\mF_{n,d}^4 &d=4.
\end{cases}
\end{equation}
The sets defined next play in this section the same role of those defined in Definition~\ref{def:partDr}.

\begin{definition}\label{def:EF}
Let $n$, $d$ and $h$ be positive integers. If $h$ is odd, let us define
\[
\begin{aligned}
&\mE_{n,d,h}^{1^*}\deq\left\{\rho\in\mathbb{D}_{k+{(h-1)}/{2}}\mid|\rho|=h-3, \rho_1\geq4\right\},\\
&\mE_{n,d,h}^{2^*}\deq\left\{\rho\in\mathbb{D}_{k+{(h+1)}/{2}}\mid|\rho|=h-1, \rho_1=1, \rho_2=2, \rho_3\geq4\right\},\\
&\mE_{n,d,h}^{3^*}\deq\left\{\rho\in\mathbb{D}_{k+{(h+1)}/{2}}\mid|\rho|=h-1, \rho_1=1, \rho_2=3 \right\},\\
&\mE_{n,d,h}^{4^*}\deq\left\{\rho\in\mathbb{D}_{k+{(h+1)}/{2}}\mid|\rho|=h-1, \rho_1=2, \rho_2=3 \right\}.
\end{aligned}
\]
If $h$ is even, let us define
\[
\begin{aligned}
&\mF_{n,d,h}^{1^*}\deq\left\{\rho\in\mathbb{D}_{k+{h}/{2}}\mid|\rho|=h-2, \rho_1=1, \rho_2\geq4\right\},\\
&\mF_{n,d,h}^{2^*}\deq\left\{\rho\in\mathbb{D}_{k+{h}/{2}}\mid|\rho|=h-2, \rho_1=2, \rho_2\geq4\right\},\\
&\mF_{n,d,h}^{3^*}\deq\left\{\rho\in\mathbb{D}_{k+{h}/{2}}\mid|\rho|=h-2, \rho_1=3, \rho_2\geq4\right\},\\
&\mF_{n,d,h}^{4^*}\deq\left\{\rho\in\mathbb{D}_{k+1+{h}/{2}}\mid|\rho|=h, \rho_1=1, \rho_2=2, \rho_3=3\right\}.
\end{aligned}
\]
\end{definition}

In the following definition we adapt the description of $\phi$ (cf.\ Definition~\ref{def:phi}) to the current representation of $\mup{T_{n,d}}^*$.
\begin{definition}\label{def:phi2}
Let us define the following correspondence from $\mup{T_{n,d}}^*\setminus \left\{ () \right\}$ to $\dist$. We will discuss later how to extend the values of the function on the empty partition  $ ()$. We denote
\[
\begin{array}{rrcl}
\phi:&\mup{T_{n,d}}^* \setminus \left\{ ()\right\} & \longrightarrow & \dist\\
&(\eta_1,\eta_2,\ldots,\eta_t)&\longmapsto&(n-2-\eta_t,\ldots,n-2-\eta_2,n-2-\eta_1).
\end{array}
\]
\end{definition}

\begin{proposition}\label{prop:oddEF}
Let $d=n-2k$ such that $4\leq d\leq n-8$ and $h\geq 4$. Then, for $1 \leq i \leq 4$, $\phi$ sends bijectively
\begin{enumerate}
\item $\mE_{n,d,h}^{i}$ into $\mE_{n,d,h}^{i^*}$,
\item $\mF_{n,d,h}^{i}$ into $\mF_{n,d,h}^{i^*}$,
\end{enumerate}
\end{proposition}
\begin{proof}
Let us prove that $\mE_{n,d,h}^{1} \overset{\phi}{\leftrightarrow}\mE_{n,d,h}^{1^*}$. The other claims can be proved in the same way.
Let us start by proving that the correspondence is well defined, i.e., if $\eta\in \mE_{n,d,h}^{1}$, then $\phi(\eta)\in \mE_{n,d,h}^{1^*}$. 
Let $\eta \in \mE_{n,d,h}^{1}$. Then, by Table~\ref{table2}, 
\[
\eta=\left(\frac{n+\left(h^2-2h-15\right)+d}{2}+i,n-h-1-i_1,\dots,n-6-i_{h-4}\right),
\]
for some positive integers $i,i_1 \geq i_2\geq\ldots \geq i_{h-4}$ such that $i=\sum_{j=1}^{h-4}i_j$. By definition of $\phi$ we have
\[
\phi(\eta)=\left(4+i_{h-4},\dots,h-1+i_1,\frac{n-h^2+2h+11-d}{2}-i\right).
\]
Notice that $|\phi(\eta)|=h-3$, $\phi(\eta)_1\geq 4$ and that $\phi(\eta)\vdash k+(h-1)/2$. Therefore $\phi(\eta)\in\mE_{n,d,h}^{1^*}$.

Notice also that $\phi$ is trivially injective and so, in order to conclude the proof, it remains to prove that $\phi$ is surjective from $\mE_{n,d,h}^{1}$ to $\mE_{n,d,h}^{1^*}$. For this purpose, let $\rho=(\rho_1,\rho_2,\ldots,\rho_{h-3})\in\mE_{n,d,h}^{1^*}$. Then the general expression for such $\rho$ is 
\[
\rho=\left(4+i_1,5+i_2,\dots,h-1+i_{h-4},k+\frac{-h^2+2h+11}{2}-i\right)
\] 
for some positive integers $i,i_1\leq i_2 \leq \ldots \leq i_{h-4}$ such that $i=\sum_{j=1}^{h-4}i_j$. It is easy to see that
\[
\eta\deq(n-2-\rho_{h-3},\ldots,n-2-\rho_2,n-2-\rho_1)
\]
is such that $\phi(\eta)=\rho$. We need to prove that $\eta\in\mE_{n,d,h}^{1}$. We have 
\begin{align*}
\eta=&\left(n-2-\left(k+\frac{-h^2+2h+11}{2}-i\right),n-2-(h-1+i_{h-4}),\ldots,n-2-(4+i_1)\right)\\
=&\left(\frac{2n-2k+h^2-2h-15}{2}+i,n-h-1-i_{h-4},\dots,n-6-i_1\right).
\end{align*}
which, from Table~\ref{table2}, is exactly the generic form of a partition in $\mE_{n,d,h}^{1}$.
\end{proof}

Notice that, as in Proposition~\ref{prop:dispBCD}, each of the sets $\mF_{n,d,h}^{i^*}$, with $h=4$, contains only one partition.
\begin{remark}\label{rmk:44}
Let $h=4$ and $d=n-2k$ be such that $4\leq d\leq n-8$. We have

\begin{enumerate}
\item $\mF_{n,d,h}^1 \overset{\phi}{\leftrightarrow} \mF_{n,d,h}^{1^*} = \left(1,k+\frac{h}{2}-1\right)$, for $d \neq 4$,
\item $\mF_{n,d,h}^2 \overset{\phi}{\leftrightarrow} \mF_{n,d,h}^{2^*} = \left(2,k+\frac{h}{2}-2\right)$, 
\item $\mF_{n,d,h}^3 \overset{\phi}{\leftrightarrow} \mF_{n,d,h}^{3^*} = \left(3,k+\frac{h}{2}-3\right)$,
\item $\mF_{n,d,h}^4 \overset{\phi}{\leftrightarrow} \mF_{n,d,h}^{4^*} = \left(1,2,3,k+\frac{h}{2}-5\right)$.
\end{enumerate}
\end{remark}

We now show how the partitions of $\mE_{n,d,\_}^{i^*}$ and $\mF_{n,d,\_}^{i^*}$ represent a convenient partition of the set $\mathbb{D}_{k+{(h-1)}/{2}}$, which will be 
used to prove the claimed bijection.

\begin{proposition}\label{prop:cardEF}
Let $d=n-2k$ such that $4\leq d\leq n-8$ and $h\geq5$ be odd. Then we have
\begin{align}\label{eq:propbije}
\mathbb{D}_{k+{(h-1)}/{2},\, h-3} =& \,\mE_{n,d,h}^{1^*}\cup\mE_{n,d,h-2}^{2^*}\cup\mE_{n,d,h-2}^{3^*}\cup\mE_{n,d,h-2}^{4^*} \, \cup \nonumber\\
 &\cup\mF_{n,d,h-1}^{1^*}\cup\mF_{n,d,h-1}^{2^*}\cup\mF_{n,d,h-1}^{3^*}\cup\mF_{n,d,h-3}^{4^*}.
\end{align}
\end{proposition}
\begin{proof}
It follows from Definition~\ref{def:EF} that each partition in one of the sets in the right side of Eq.~\eqref{eq:propbije} is a partition of $k+{(h-1)}/{2}$ into $h-3$ distinct parts, therefore we have 
\begin{align*}
\mathbb{D}_{k+{(h-1)}/{2},\, h-3} \supseteq & \,\mE_{n,d,h}^{1^*}\cup\mE_{n,d,h-2}^{2^*}\cup\mE_{n,d,h-2}^{3^*}\cup\mE_{n,d,h-2}^{4^*} \, \cup\\
 &\cup\mF_{n,d,h-1}^{1^*}\cup\mF_{n,d,h-1}^{2^*}\cup\mF_{n,d,h-1}^{3^*}\cup\mF_{n,d,h-3}^{4^*}.
\end{align*}
To prove the converse, it is enough to notice that the claimed sets form a partition of the set  $\mathbb{D}_{k+{(h-1)}/{2},\, h-3}$, indeed we can write 
\[
\begin{aligned}
\mE_{n,d,h}^{1^*}=&\left\{\l\in \him\mid \l_1\geq4\right\},\\
\mF_{n,d,h-1}^{1^*}=&\left\{\l\in\him\mid \l_1=1,\l_2\geq4\right\},\\
\mF_{n,d,h-1}^{2^*}=&\left\{\l\in\him\mid \l_1=2,\l_2\geq4\right\},\\
\mF_{n,d,h-1}^{3^*}=&\left\{\l\in\him \mid \l_1=3,\l_2\geq4\right\},\\
\mE_{n,d,h-2}^{2^*}=&\left\{\l\in\him \mid \l_1=1,\l_2=2,\l_3\geq4\right\},\\
\mE_{n,d,h-2}^{3^*}=&\left\{\l\in\him\mid \l_1=1,\l_2=3,\l_3\geq4\right\},\\
\mE_{n,d,h-2}^{4^*}=&\left\{\l\in\him\mid \l_1=2,\l_2=3,\l_3\geq4\right\},\\
\mF_{n,d,h-3}^{4^*}=&\left\{\l\in\him \mid \l_1=1,\l_2=2,\l_3=3\right\}.
\end{aligned}
\]
\end{proof}

We now use Proposition~\ref{prop:cardEF} to show the claimed bijection related to the case $\l_t=2n-5$ of the main result introduced in Sec.~\ref{sec:contri}.
\begin{theorem}\label{thm_main2}
Let $d=n-2k$ such that $4 \leq d\leq n-8$. Then 
\[
\mup{T_{n,d}} \leftrightarrow \dpok.
\]
\end{theorem}
\begin{proof}
Let us start assuming $d > 4$. We obtain the claim by showing first that $\mup{T_{n,d}} \leftrightarrow \bigcup_{i\geq 0} \mathbb{D}_{k+2+i,2+2i}$ and successively
that $\bigcup_{i\geq 0} \mathbb{D}_{k+2+i,2+2i} \leftrightarrow \dpok$. The first claim follows directly from Proposition~\ref{prop:cardEF}, indeed

\[
\begin{aligned}
\mup{T_{n,d}} \leftrightarrow &\mup{T_{n,d}}^*\\
=&\, \mE_{n,d}\cup\mF_{n,d}\\
=&\bigcup_{h\geq5}\mE_{n,d,h}^{1} \cup \bigcup_{h\geq5}\mE_{n,d,h}^{2} \cup \bigcup_{h\geq5}\mE_{n,d,h}^{3} \cup \bigcup_{h\geq5}\mE_{n,d,h}^{4} \cup \\
&\cup \bigcup_{h\geq4}\mF_{n,d,h}^{1} \cup \bigcup_{h\geq4}\mF_{n,d,h}^{2} \cup \bigcup_{h\geq4}\mF_{n,d,h}^{3} \cup \bigcup_{h\geq4}\mF_{n,d,h}^{4}\\
\leftrightarrow& \bigcup_{h\geq5}\mE_{n,d,h}^{1^*}\cup\bigcup_{h\geq5}\mE_{n,d,h}^{2^*}\cup\bigcup_{h\geq5}\mE_{n,d,h}^{3^*}\cup\bigcup_{h\geq5}\mE_{n,d,h}^{4^*}\cup\\
&\cup\bigcup_{h\geq4}\mF_{n,d,h}^{1^*}\cup\bigcup_{h\geq4}\mF_{n,d,h}^{2^*}\cup\bigcup_{h\geq4}\mF_{n,d,h}^{3^*}\cup\bigcup_{h\geq4}\mF_{n,d,h}^{4^*}\\
=&\bigcup_{\substack{h\geq5\\h \text { odd }}} \left(\mE_{n,d,h}^{1^*}\cup\mE_{n,d,h-2}^{2^*}\cup\mE_{n,d,h-2}^{3^*}\cup\mE_{n,d,h-2}^{4^*}  \right) \cup\\
& \bigcup_{\substack{h\geq5\\h \text { odd }}}\left(\mF_{n,d,h-1}^{1^*}\cup\mF_{n,d,h-1}^{2^*}\cup\mF_{n,d,h-1}^{3^*}\cup\mF_{n,d,h-3}^{4^*}\right)\\
=&\bigcup_{\substack{h\geq5\\h \text { odd }}}\him\\
=&\bigcup_{i\geq 0} \mathbb{D}_{k+2+i,\,2+2i}.
\end{aligned}
\]
Notice that the union in the last equation does not provide any contribution when  $i$ is sufficiently large, therefore it represents a finite union of sets. It can be noticed indeed that the largest number of parts that can appear in a partition of $T_{n,d}$ is approximatively the square root of $n$, while there is a linear dependence in $i$ between $k+2+i$ and $2+2i$.

Let us now prove that $\bigcup_{i\geq 0} \mathbb{D}_{k+2+i,2+2i} \leftrightarrow \dpok$. First notice that, if $\l \in \dpok$, then $|\l|$ is even, therefore the following equation trivially holds
\[
\dpok = \bigcup_{i\geq 0}\mathbb{D}_{2\left(k+1\right),\,2+2i}^{\,\text{odd }},
\]
where the last union is again only formally infinite. 
Let us define 
\[
\begin{tabular}{cccc}
$\psi\colon$ & $\mathbb{D}_{k+2+i,\,2+2i}$ & $\to$ & $\mathbb{D}_{2\left(k+1\right),\,2+2i}^{\,\text{odd}}$\\
& $\left(\l_1,\dots,\l_{2+2i}\right)$ & & $\left(2\l_1-1,\dots,2\l_{2+2i}-1\right)$
\end{tabular}
\]
and let us prove that $\psi$ is bijective.
Clearly $\psi$ is well defined, indeed if $\l \in \mathbb{D}_{k+2+i,\,2+2i}$, then 
\[
\begin{aligned}
\psi\left(\l\right)\vdash 2\l_1-1+\dots+2\l_{2+2i}-1&=2\left(\l_1+\dots+\l_{2+2i}\right)-2-2i\\
&=2\left(k+2+i\right)-2-2i\\
&=2k+2.
\end{aligned}
\]
Let us now prove that $\psi$ is surjective. Let $\sigma = \left(\sigma_1, \sigma_2, \dots,\sigma_{2+2i}\right) \in \mathbb{D}_{2\left(k+1\right),\,2+2i}^{\,\text{odd}}$.
It is easy to verify that 
\[
\rho \deq \left(\frac{\sigma_1+1}{2},\frac{\sigma_2+1}{2},\dots,\frac{\sigma_{2+2i}+1}{2}\right)
\]
is such that $\psi(\rho) = \sigma$.
Since $\psi$ is trivially injective, the claim is proved for $d > 4$.
In the case $d=4$, from Eq.~\eqref{eq:ud4} we have

\begin{align*}
\mup{T_{n,4}} \leftrightarrow &\mup{T_{n,4}}^*\\
=&\, \mE_{n,4}\,\dot\cup\,\left\{\left(\right)\right\}\,\dot\cup\,\left(\bigcup_{h\geq6}\mF_{n,4,h}^1\right)\,\dot\cup\,\mF_{n,4}^2\,\dot\cup\,\mF_{n,4}^3\,\dot\cup\,\mF_{n,4}^4.
\end{align*}
The claim is obtained as before, only noticing the empty partition $()$ replaces the partition  of $\mF_{n,d,4}^1$, which is not defined when $d=4$ (cf. Remark~\ref{rmk:44}).

\end{proof}
\section{Conclusions and open problems}\label{sec_final}
In this paper we completed the classification of maximal unrefinable partitions started by Aragona et al.~\cite{aragona2021maximal}. Now we have that, if $N$ is the triangular number $T_n$, then the number of maximal unrefinable partitions of $T_n$ is one if $n$ is even and coincides with the number of partitions
 of $(n+1)/2$ into distinct parts if $n$ is odd. If $N$ is non-triangular, i.e., if $N=T_{n,d}$ for some $n \geq 11$ and $1 \leq d \leq n-1$, from Theorem~\ref{thm:maindodd} and Theorem~\ref{thm_main2} we obtain:

\begin{corollary}
%Let $n\geq 11$ and $1\leq d \leq n-1 $ be integers. 
If $n$ is odd, then
\[
\#\mup{T_{n,d}}=
\begin{cases}
1+\#\dist_{(n-d+1)/2} & \text{ if } d  > 3\text{ is even},\\
\# {\mathbb{D}_{n-d+2}^{\,\text {odd }}}& \text{ if } d > 3 \text{ is odd},\\
1 & \text{ if } d \in \{1,2,3\}.
\end{cases}
\]
Otherwise
\[
\#\mup{T_{n,d}}=
\begin{cases}
1+\#\dist_{(n-d+1)/2} & \text{ if } d >2 \text{ is odd},\\
\# {\mathbb{D}_{n-d+2}^{\,\text {odd }}}& \text{ if } d>2 \text{ is even}, \\
1 & \text{ if } d \in \{1,2\}.
\end{cases}
\]
\end{corollary}

The two results are illustrated in Fig.~\ref{fig:final}, where we list the number of maximal unrefinable partitions for integers included between two consecutive triangular numbers. Precisely, we start from an even integer $n$ and list the number $\#\mup{T_{n,d}}$ and the corresponding maximum $\l_t$, for each integer in 
$\{s \in \mathbb N \mid T_{n-1} \leq s \leq T_{n+1}\}$. The same combinatorial structure replicates in other intervals between two consecutive triangular numbers, according to the rules of Theorem~\ref{thm:maindodd} and Theorem~\ref{thm_main2}.

\subsection*{Open problems}
The classification of maximal unrefinable partitions has been achieved constructively, by enumerating all the possibilities. It is not clear to the authors if there exists a more concise 
way to prove the result by means of non-constructive arguments. Moreover, it remains an open question whether results of the same nature, i.e., showing that maximal unrefinable partitions are \emph{de facto} partitions into distinct parts, applies also when removing the hypothesis of maximality. To our knowledge, very little is known in this sense regarding unrefinable partitions.

\section*{Acknowledgements}
R. Aragona and R. Civino are members of INdAM-GNSAGA
 (Italy). R. Civino is funded by the Centre of excellence
 ExEMERGE at the University of L'Aquila. The authors gratefully acknowledge financial support from MUR–Italy through PRIN 2022RFAZCJ “Algebraic Methods in Cryptanalysis”, with full funding provided for L.~Campioni.

\bibliographystyle{amsalpha}
\bibliography{sym2n_ref.bib}

% \clearpage 
  
 \begin{figure}
\[
\begin{tabular}{|c|c| c ||c|c|c|}
\hline\xrowht[()]{8pt}
$N$ & ${\l_t}_{\text{max}}$ & $\#\mup{N}$ & $N$ & ${\l_t}_{\text{max}}$ & $\#\mup{N}$\\
\hline
\hline\xrowht[()]{8pt}
$T_{n-1}$ & $2n-4$ & $\mathbb{D}_{{n}/{2}}$ & $T_n$ & $2n-4$ & $1$\\
\hline\xrowht[()]{8pt}
$T_{n,n-1}$ & $2n-4$ & $1$ & $T_{n+1,n}$ & $2n-4$ & $1$\\
\hline\xrowht[()]{8pt}
$T_{n,n-2}$ & $2n-5$ & $1$ & $T_{n+1,n-1}$ & $2n-5$ & $1$\\
\hline\xrowht[()]{8pt}
$T_{n,n-3}$ & $2n-4$ & $1$ & $T_{n+1,n-2}$ & $2n-4$ & $1$\\
\hline\xrowht[()]{8pt}
$T_{n,n-4}$ & $2n-5$ & $1$ & $T_{n+1,n-3}$ & $2n-5$ & $1$\\
\hline\xrowht[()]{8pt}
$T_{n,n-5}$ & $2n-4$ & $2$ & $T_{n+1,n-4}$ & $2n-4$ & $2$\\
\hline\xrowht[()]{8pt}
$T_{n,n-6}$ & $2n-5$ & $2$ & $T_{n+1,n-5}$ & $2n-5$ & $2$\\
\hline\xrowht[()]{8pt}
\vdots & \vdots & \vdots & \vdots & \vdots & \vdots\\
\hline\xrowht[()]{8pt}
$T_{n,n-\left(2k-1\right)}$ & $2n-4$ & $\mathbb{D}_{k}+1$ & $T_{n+1,n+1-\left(2k-1\right)}$ & $2n-4$ & $\mathbb{D}_{k}+1$\\
\hline\xrowht[()]{8pt}
$T_{n,n-\left(2k\right)}$ & $2n-5$ & $\mathbb{D}_{2k+2}^{\,\text{odd}}$ & $T_{n+1,n+1-\left(2k\right)}$ & $2n-5$ & $\mathbb{D}_{2k+2}^{\,\text{odd}}$\\
\hline\xrowht[()]{8pt}
\vdots & \vdots & \vdots & \vdots & \vdots & \vdots\\
\hline\xrowht[()]{8pt}
$T_{n,4}$ & $2n-5$ & $\mathbb{D}_{n-2}^{\,\text{odd}}$ & $T_{n+1,5}$ & $2n-5$ & $\mathbb{D}_{n-2}^{\,\text{odd}}$\\
\hline\xrowht[()]{8pt}
$T_{n,3}$ & $2n-4$ & $\mathbb{D}_{{(n-2)}/{2}}$ & $T_{n+1,4}$ & $2n-4$ & $\mathbb{D}_{{(n-2)}/{2}}$\\
\hline\xrowht[()]{8pt}
$T_{n,2}$ & $2n-3$ & $1$ & $T_{n+1,3}$ & $2n-4$ & $1$\\
\hline\xrowht[()]{8pt}
$T_{n,1}$ & $2n-2$ & $1$ & $T_{n+1,2}$ & $2n-3$ & $1$\\
\hline\xrowht[()]{8pt}
$T_n$ & $2n-4$ & $1$ & $T_{n+1,1}$ & $2n-2$ & $1$\\
\hline\xrowht[()]{8pt}
$\star$&$\star$&$\star$&  $T_{n+1}$ & $2n-4$ & $\mathbb{D}_{\,1+{n}/{2}}$ \\
\hline
\end{tabular}
\]
\caption{The number of maximal unrefinable partitions between two consecutive triangular numbers. Here $n$ is an even number.}
\label{fig:final}
\end{figure}
 
\begin{landscape}

\begin{table}[h!]
\input{tab1.tex}
\medskip

\caption{List of all the possible maximal constructions when $\l_t=2n-4$.}
\label{table1}
\end{table}

\end{landscape}

\begin{landscape}
\begin{table}[]
\input{tab2.tex}
\medskip

\caption{List of all the possible maximal constructions when $\l_t=2n-5$.}
\label{table2}
\end{table}
\end{landscape}

\end{document}

%% file: preamble.tex
\usepackage{amssymb,amsmath}
\usepackage[english]{babel}
\usepackage{amsfonts}
\usepackage{color}
\usepackage{amsthm}
\usepackage[colorlinks=true,linkcolor=NavyBlue, citecolor=NavyBlue,urlcolor=NavyBlue]{hyperref}
\usepackage{graphicx}
\usepackage{mathtools}
\usepackage[usenames,dvipsnames]{pstricks}
\usepackage{bm}
\usepackage{amsmath}
\usepackage{listings}
\usepackage{hyperref}
\usepackage[usenames,dvipsnames]{pstricks}
 \usepackage{epsfig}
 \usepackage{pst-grad} % For gradients
 \usepackage{pst-plot} % For axes
\usepackage{xy}
\usepackage[draft]{fixme}
\DeclareMathOperator{\Sym}{Sym}

\newcommand\deq{\mathrel{\stackrel{\makebox[0pt]{\mbox{\normalfont\tiny def}}}{=}}}

\theoremstyle{plain}
\newtheorem{theorem}{Theorem}[section]
\newtheorem*{notheorem}{Theorem}
\newtheorem*{mainres}{Main theorem}
\newtheorem{lemma}[theorem]{Lemma}
\newtheorem{proposition}[theorem]{Proposition}
\newtheorem{corollary}[theorem]{Corollary}

\theoremstyle{remark}
\newtheorem{remark}{Remark}

\theoremstyle{definition}
\newtheorem{definition}[theorem]{Definition}

\newtheorem*{notation*}{Notation}
\usepackage{listings}
\lstdefinelanguage{GAP}{%
 morekeywords={%
 Assert,Info,IsBound,QUIT,%
 TryNextMethod,Unbind,and,break,%
 continue,do,elif,%
 else,end,false,fi,for,%
 function,if,in,local,%
 mod,not,od,or,%
 quit,rec,repeat,return,%
 then,true,until,while%
 },%
 sensitive,%
 morecomment=[l]\#,%
 morestring=[b]",%
 morestring=[b]',%
}[keywords,comments,strings]

\usepackage[T1]{fontenc}
\usepackage[variablett]{lmodern}
\usepackage{xcolor}
\usepackage[foot]{amsaddr}
\usepackage{enumerate}
\usepackage{arydshln}
 \usepackage{tikz}
\usepackage{lscape} 
\let \l \lambda

\newcommand{\listP}{\l = (\l_1,\l_2, \dots, \l_t)}
\newcommand{\missP}{\mu_1 < \mu_2 < \dots < \mu_m}
\newcommand{\Ml}{\mathcal M_\l}

\newcommand{\al}{\alpha}
\newcommand{\mup}[1]{\widetilde{\mathbb{U}}_{#1}}
\newcommand{\up}{\mathbb{U}}

\newcommand{\su}{\sum a_i = \sum \al_i}

\newcommand{\dist}{\mathbb D}
\newcommand{\mA}{\mathcal{A}}
\newcommand{\mB}{\mathcal{B}}
\newcommand{\mC}{\mathcal{C}}
\newcommand{\mD}{\mathcal{D}}

\newcommand{\mE}{\mathcal{E}}
\newcommand{\mF}{\mathcal{F}}

\newcommand{\dpok}{\mathbb{D}_{2(k+1)}^{\,\text {odd }}}
\newcommand{\him}{\mathbb{D}_{k+{(h-1)}/{2},h-3}}

\usepackage{stackengine}
\newcommand\xrowht[2][0]{\addstackgap[.5\dimexpr#2\relax]{\vphantom{#1}}}

%% file: tab1.tex
\begin{tabular}{|c|c|c|}
    \hline
     $d$ & $\left(a_1,\dots,a_h\right)$ & $\left(\al_1,\dots,\al_j\right)$\\
\hline
     \hline
     $3$ & $\left(n-1,n\right)$ & $\left(d,2n-4\right)$\\
     \hline
     $n-1$ & $\left(n-3,n-2,n\right)$ & $\left(d,2n-4\right)$\\
     \hline
     $n-3$ & $\left(n-4,n-2,n-1\right)$ & $\left(d,2n-4\right)$\\
     \hline
     $n-5$ & $\left(n-4,n-3,n-2\right)$ & $\left(d,2n-4\right)$\\
     \hline
     $n-5$ & $\left(n-2,n-1,n\right)$ & $\left(n+1,2n-4\right)$\\
     \hline
     $3<n-\left(2k-1\right)\leq n-7$ & $\left(\frac{n-5+d}{2},n-2,n-1,n\right)$ & $\left(d,\frac{3n-3-d}{2},2n-4\right)$\\
     \hline
     $3\leq n-\left(2k-1\right)\leq n-7$ & $\left(\frac{n-3+d}{2},n-3,n-2,n\right)$ & $\left(d,\frac{3n-5-d}{2},2n-4\right)$\\
     \hline
     $3\leq n-\left(2k-1\right)\leq n-9$ & $\left(\frac{n-1+d}{2},n-4,n-2,n-1\right)$ & $\left(d,\frac{3n-7-d}{2},2n-4\right)$\\
     \hline
     $3\leq n-\left(2k-1\right)\leq n-11$ & $\left(\frac{n+1+d}{2},n-4,n-3,n-2\right)$ & $\left(d,\frac{3n-9-d}{2},2n-4\right)$\\
     \hline
     
     $3\leq n-\left(2k-1\right)\leq n-\left(h^2-h-7\right)$ & $\Big{(}\frac{n+\left(h^2-3h-9\right)+d}{2}+i,n-h-i_1,\dots$ & $(d,n+1+i_{h-4},\dots,n-4+h+i_1,$\\
     	for $h\geq5$ & $n-5-i_{h-4},n-2,n-1,n)$ & $\frac{3n-\left(h^2-3h-1\right)-d}{2}-i,2n-4\Big{)}$\\
     \hline
     
     $3\leq n-\left(2k-1\right)\leq n-\left(h^2-h-5\right)$ & $\Big{(}\frac{n+\left(h^2-3h-7\right)+d}{2}+i,n-h-i_1,\dots$ & $(d,n+1+i_{h-4},\dots,n-4+h+i_1,$\\
    for  $h\geq5$ & $\dots,n-5-i_{h-4},n-3,n-2,n)$ & $\frac{3n-\left(h^2-3h+1\right)-d}{2},2n-4\Big{)}$\\
     \hline
     
     $3\leq n-\left(2k-1\right)\leq n-\left(h^2-h-3\right)$ & $\Big{(}\frac{n+\left(h^2-3h-5\right)+d}{2}+i,n-h-i_1,$ & $(d,n+1+i_{h-4},\dots,n-4+h+i_1,$\\
     for $h\geq5$ & $\dots,n-5-i_{h-4},n-4,n-2,n-1)$ & $\frac{3n-\left(h^2-3h+3\right)-d}{2},2n-4\Big{)}$\\
     \hline
     
     $3\leq n-\left(2k-1\right)\leq n-\left(h^2-h-1\right)$ & $\Big{(}\frac{n+\left(h^2-3h-3\right)+d}{2}+i,n-h-i_1,$ & $(d,n+1+i_{h-4},\dots,n-4+h+i_1,$\\
     for $h\geq5$ & $\dots,n-5-i_{h-4},n-4,n-3,n-2)$ & $\frac{3n-\left(h^2-3h+5\right)-d}{2},2n-4\Big{)}$\\
     \hline
\end{tabular}

%% file: tab2.tex
\begin{tabular}{|c|c|c|}
    \hline
     $d$ & $\left(a_1,\dots,a_h\right)$ & $\left(\al_1,\dots,\al_j\right)$\\
     \hline
     \hline
     $4$ & $\left(n-1,n\right)$ & $\left(d,2n-5\right)$\\
     \hline
     $n-2$ & $\left(n-4,n-3,n\right)$ & $\left(d,2n-5\right)$\\
     \hline
     $n-4$ & $\left(n-5,n-3,n-1\right)$ & $\left(d,2n-5\right)$\\
     \hline
     $n-6$ & $\left(n-5,n-4,n-2\right)$ & $\left(d,2n-5\right)$\\
     \hline
     $n-6$ & $\left(n-3,n-1,n\right)$ & $\left(n+1,2n-5\right)$\\
     \hline
     $4<n-2k\leq n-8$ & $\left(\frac{n-6+d}{2},n-3,n-1,n\right)$ & $\left(d,\frac{3n-4-d}{2},2n-5\right)$\\
     \hline
     $4\leq n-2k\leq n-8$ & $\left(\frac{n-4+d}{2},n-4,n-2,n\right)$ & $\left(d,\frac{3n-6-d}{2},2n-5\right)$\\
     \hline
     $4\leq n-2k\leq n-10$ & $\left(\frac{n-2+d}{2},n-5,n-2,n-1\right)$ & $\left(d,\frac{3n-8-d}{2},2n-5\right)$\\
     \hline
     $4\leq n-2k\leq n-14$ & $\left(\frac{n+2+d}{2},n-5,n-4,n-3\right)$ & $\left(d,\frac{3n-12-d}{2},2n-5\right)$\\
     \hline
     
     $4\leq n-2k\leq n-\left(h^2-11\right)$ & $\Big{(}\frac{n+\left(h^2-2h-15\right)+d}{2}+i,n-h-1-i_1,\dots$ & $(d,n+1+i_{h-4},\dots,n-4+h+i_1,$\\
     for $h\geq5,h\text{ odd }$ & $n-6-i_{h-4},n-2,n-1,n)$ & $\frac{3n-\left(h^2-2h-5\right)-d}{2}-i,2n-5\Big{)}$\\
     \hline
     
     $4\leq n-2k\leq n-\left(h^2-10\right)$ & $\Big{(}\frac{n+\left(h^2-2h-14\right)+d}{2}+i,n-h-1-i_1,\dots$ & $(d,n+1+i_{h-4},\dots,n-4+h+i_1,$\\
     for $h\geq5,h \text{ even }$ & $n-6-i_{h-4},n-3,n-1,n)$ & $\frac{3n-\left(h^2-2h-4\right)-d}{2}-i,2n-5\Big{)}$\\
     \hline

     $4\leq n-2k\leq n-\left(h^2-8\right)$ & $\Big{(}\frac{n+\left(h^2-2h-12\right)+d}{2}+i,n-h-1-i_1,\dots$ & $(d,n+1+i_{h-4},\dots,n-4+h+i_1,$\\
   for  $h\geq5,h \text{ even }$ & $n-6-i_{h-4},n-4,n-2,n)$ & $\frac{3n-\left(h^2-2h-2\right)-d}{2}-i,2n-5\Big{)}$\\
     \hline
     
     $4\leq n-2k\leq n-\left(h^2-7\right)$ & $\Big{(}\frac{n+\left(h^2-2h-11\right)+d}{2}+i,n-h-1-i_1,\dots$ & $(d,n+1+i_{h-4},\dots,n-4+h+i_1,$\\
     for $h\geq5,h \text{ odd }$ & $n-6-i_{h-4},n-4,n-3,n)$ & $\frac{3n-\left(h^2-2h-1\right)-d}{2}-i,2n-5\Big{)}$\\
     \hline

     $4\leq n-2k\leq n-\left(h^2-6\right)$ & $\Big{(}\frac{n+\left(h^2-2h-10\right)+d}{2}+i,n-h-1-i_1,\dots$ & $(d,n+1+i_{h-4},\dots,n-4+h+i_1,$\\
     for $h\geq5,h \text{ even }$ & $n-6-i_{h-4},n-5,n-2,n-1)$ & $\frac{3n-\left(h^2-2h\right)-d}{2}-i,2n-5\Big{)}$\\
     \hline
     
     $4\leq n-2k\leq n-\left(h^2-5\right)$ & $\Big{(}\frac{n+\left(h^2-2h-9\right)+d}{2}+i,n-h-1-i_1,\dots$ & $(d,n+1+i_{h-4},\dots,n-4+h+i_1,$\\
     for $h\geq5,h \text{ odd }$ & $n-6-i_{h-4},n-5,n-3,n-1)$ & $\frac{3n-\left(h^2-2h+1\right)-d}{2}-i,2n-5\Big{)}$\\
     \hline
     
     $4\leq n-2k\leq n-\left(h^2-3\right)$ & $\Big{(}\frac{n+\left(h^2-2h-7\right)+d}{2}+i,n-h-1-i_1,\dots$ & $(d,n+1+i_{h-4},\dots,n-4+h+i_1,$\\
    for  $h\geq5,h \text{ odd }$ & $n-6-i_{h-4},n-5,n-4,n-2)$ & $\frac{3n-\left(h^2-2h+3\right)-d}{2}-i,2n-5\Big{)}$\\
     \hline
     
     $4\leq n-2k\leq n-\left(h^2-2\right)$ & $\Big{(}\frac{n+\left(h^2-2h-6\right)+d}{2}+i,n-h-1-i_1,\dots$ & $(d,n+1+i_{h-4},\dots,n-4+h+i_1,$\\
     for $h\geq5,h \text{ even }$ & $n-6-i_{h-4},n-5,n-4,n-3)$ & $\frac{3n-\left(h^2-2h+4\right)-d}{2}-i,2n-5\Big{)}$\\
     \hline
     
\end{tabular}